\documentclass[]{amsart}
\usepackage{enumerate}
\usepackage{amsmath}
\usepackage{amssymb}
\usepackage{graphicx}

\newtheorem{theorem}{Theorem}[section]
\newtheorem{lemma}[theorem]{Lemma}
\newtheorem{corollary}[theorem]{Corollary}

\newtheorem{remark}[theorem]{Remark}
\newtheorem{example}[theorem]{Example}
\numberwithin{equation}{section}

\begin{document}

\title[On the quantization for self-affine measures]{On the quantization for self-affine measures on Bedford-McMullen
carpets}
\author{Marc Kesseb\"{o}hmer}
\address{Marc Kesseb\"ohmer, Fachbereich 3 -- Mathematik und Informatik, Universit\"{a}t Bremen, Bibliothekstr. 1, Bremen 28359, Germany}
\email{mhk@math.uni-bremen.de}
\author{Sanguo Zhu}
\address{Sanguo Zhu, School of Mathematics and Physics, Jiangsu University of Technology,
Changzhou 213001, China}
\email[Corresponding author]{sgzhu@jsut.edu.cn}

\begin{abstract}
For a self-affine measure  on a Bedford-McMullen carpet we prove
that its quantization dimension exists and determine its exact value. Further, we give various
sufficient conditions for the corresponding upper and lower quantization
coefficient  to be both positive and finite. Finally, we
compare the quantization dimension with corresponding quantities derived
from the multifractal temperature function and show that -- different
from conformal systems -- they in general do not coincide.
\end{abstract}

 \maketitle

\section{Introduction and statement of results}

The quantization problem for probability measures has its origin in
information theory and engineering technology (cf. \cite{BW:82,GN:98,Za:63}).
Mathematically, the problem of determining the asymptotic error in
the approximation of a given probability measure by discrete probability
measures with finite support in terms of $L_{r}$-metrics
is addressed. We refer to \cite{GL:00} for rigorous mathematical
foundations of quantization theory, Further related results can be
found in \cite{GL:01,GL:04,GL:05,LM:02,Kr:08,PK:01}.

Let $\|\cdot\|$ be a norm on $\mathbb{R}^{q}$ and $d$ the metric
induced by this norm. For each $k\in\mathbb{N}$, we write $\mathcal{D}_{k}:=\{\alpha\subset\mathbb{R}^{q}:1\leq{\rm card}(\alpha)\leq k\}$.
Let $\nu$ be a Borel probability measure on $\mathbb{R}^{q}$. The
$k$\emph{th quantization error} for $\nu$ of order $r$ is defined \cite{GL:00} by
\begin{eqnarray}
e_{k,r}(\nu):=\left\{ \begin{array}{ll}
\inf_{\alpha\in\mathcal{D}_{k}}\big(\int d(x,\alpha)^{r}d\nu(x)\big)^{1/r},\;\;\;\;\;\; r>0,\\
\inf_{\alpha\in\mathcal{D}_{k}}\exp\int\log d(x,\alpha)d\nu(x),\;\;\;\;\; r=0.
\end{array}\right.\label{quanerror}
\end{eqnarray}
A set $\alpha\subset\mathbb{R}^{q}$ is called an $k$-\emph{optimal
set} (of order $r$) for $\nu$ if $1\leq{\rm card}(\alpha)\leq k$
and the infimum in (\ref{quanerror}) is attained at $\alpha$. The
collection of all the $k$-optimal sets of order $r$ is denoted by
$C_{k,r}(\nu)$. According to \cite{GL:04}, under some natural conditions,
$e_{k,r}(\nu)$ tends to $e_{k,0}(\nu)$ as $r$ tends to zero. We
also call $e_{k,0}(\nu)$ the $k$\emph{th geometric mean error}
for $\nu$. So the $k$th geometric mean error is a limiting case
of the $L_{r}$-quantization error $e_{k,r}(\nu)$ when $r\to0$.
To characterize the speed at which the quantization error $e_{k,r}(\nu)$
tends to zero as $k$ increases to infinity, we consider the \emph{upper}
and \emph{lower quantization dimension} for $\nu$ of order $r\in[0,\infty)$
\cite{GL:00,GL:04}:
\begin{eqnarray}
\overline{D}_{r}(\nu):=\limsup_{k\to\infty}\frac{\log k}{-\log e_{k,r}(\nu)},\;\underline{D}_{r}(\nu):=\liminf_{k\to\infty}\frac{\log k}{-\log e_{k,r}(\nu)}.\label{quandimdef}
\end{eqnarray}
If $\overline{D}_{r}(\nu)=\underline{D}_{r}(\nu)$, the common value
is called the \emph{quantization dimension} for $\nu$ of order $r$
and denoted by $D_{r}(\nu)$. Compared with the upper and lower quantization
dimension, the $s$\emph{-dimensional upper} and\emph{ lower quantization
coefficient}
\begin{eqnarray*}
\underline{Q}_{r}^{s}(\nu):=\liminf_{k\to\infty}k^{1/s}e_{k,r}(\nu),\;\;\overline{Q}_{r}^{s}(\nu):=\limsup_{k\to\infty}k^{1/s}e_{k,r}(\nu),\;\; s\in(0,\infty)
\end{eqnarray*}
provide us with some more accurate information on the asymptotic properties
of the quantization error, given that they are both positive and finite.

Now, to introduce self-affine measures on Bedford-McMullen carpets,
fix two integers $m,n$ with $m\leq n$ and fix a set $G\subset\big\{0,1,\ldots,n-1\big\}\times\big\{0,1,\ldots,m-1\big\}$
with $N:=\mbox{card}\left(G\right)\geq2$. We define
a family of affine mappings on $\mathbb{R}^{2}$ by
\begin{equation}
f_{ij}:(x,y)\mapsto\big(n^{-1}x+n^{-1}i,m^{-1}y+m^{-1}j\big),\;\;(i,j)\in G.\label{fi's}
\end{equation}
By a result of Hutchinson \cite{Hut:81} there exists a unique non-empty
compact set $E$ satisfying $E=\bigcup_{(i,j)\in G}^{N}f_{ij}(E)$,
which is a special case of a \emph{self-affine set} called the\emph{
Bedford-McMullen carpet }determined by $(f_{ij})_{(i,j)\in G}$. Furthermore,
for a fixed probability vector $(p_{ij})_{(i,j)\in G}$ with $p_{ij}>0$,
for all $(i,j)\in G$, there exists a unique Borel probability measure
$\mu$ supported on $E$ satisfying
\begin{equation}
\mu=\sum_{(i,j)\in G}p_{ij}\mu\circ f_{ij}^{-1},\label{selfaffinemeas}
\end{equation}
which we call the \emph{self-affine measure} associated with $(p_{ij})_{(i,j)\in G}$
and $(f_{ij})_{(i,j)\in G}$. Sets and measures of this form have
been intensively studied in the past decades, see e.g. \cite{Bed:84,Mcmullen:84,LG:92,Peres:94b,King:95,Fal:10,GLi:10}
for many interesting results. Throughout the paper, $\mu$
will denote such a self-affine measure on a Bedford-McMullen carpet
and we are going to focus on the quantization problem associated to
such measures. Let us set $\theta:=\log m/\log n$ and write
\begin{eqnarray*}
&&G_{x} =\left\{ i:(i,j)\in G\;{\rm for\; some\;}j\right\},\\
&&G_{y}=\left\{ j:(i,j)\in G\;{\rm for\; some\;}i\right\} ,\\
&&G_{x,j} =\left\{ i:(i,j)\in G\right\},\;\;q_{j}:=\sum_{i\in G_{x,j}}p_{ij} .
\end{eqnarray*}
Whenever we consider the geometric mean error, i.e. for $r=0$, due
to some technical reasons, we will additionally assume that
\begin{equation}
\min_{i_{1},i_{2}\in G_{x}}|i_{1}-i_{2}|\geq1,\;\min_{j_{1},j_{2}\in G_{y}}|j_{1}-j_{2}|\geq1.\label{hyposep}
\end{equation}
We are now in the position to state our main result.
\begin{theorem}
\label{mthm1} Let $\mu$ be self-affine measure on a Bedford-McMullen
carpet. Then for each $r\geq0$ we have that $D_{r}(\mu)$ exists
and equals $s_{r}$, where
\begin{eqnarray}
s_{0}:=(-\log m)^{-1}\bigg(\theta\sum_{(i,j)\in G}p_{ij}\log p_{ij}+(1-\theta)\sum_{j\in G_{y}}q_{j}\log q_{j}\bigg)\label{s0}
\end{eqnarray}
and for $r>0$ the number $s_{r}$ is given by the unique solution
of
\begin{eqnarray}
m^{-\frac{rs_{r}}{s_{r}+r}}\bigg(\sum_{(i,j)\in G}p_{ij}^{\frac{s_{r}}{s_{r}+r}}\bigg)^{\theta}\bigg(\sum_{j\in G_{y}}q_{j}^{\frac{s_{r}}{s_{r}+r}}\bigg)^{1-\theta}=1.\label{maineq1}
\end{eqnarray}
Moreover, the quantization coefficients of order $r$ are finite and positive, i.e.,
\begin{eqnarray}
0<\underline{Q}_{r}^{s_{r}}(\mu)\leq\overline{Q}_{r}^{s_{r}}(\mu)<\infty\label{main1}
\end{eqnarray}
if one of the following conditions is fulfilled:

\begin{enumerate}[(a)]
\item  \label{ConditionA} $r>0$ and $C_{j,r}:=q_{j}^{-\frac{s_{r}}{s_{r}+r}}\sum_{i\in G_{x,j}}p_{ij}^{\frac{s_{r}}{s_{r}+r}}$
are identical for all $j\in G_{y}$,
\item  \label{ConditionB} $r=0$ and $C_{j}:=q_{j}^{-1}\sum_{i\in G_{x,j}}p_{ij}\log\left(p_{ij}/q_{j}\right)$
are identical for all $j\in G_{y}$,
\item  \label{ConditionC} $r\geq0$ and $q_{j}$
are identical for all $j\in G_{y}$.

\end{enumerate}\end{theorem}
\begin{remark}{\rm
We would like to remark that the existence of the quantization dimension
of order zero and its value can be deduced from some general considerations
as follows. As is noted in \cite{GL:04}, the asymptotic geometric
mean error for a Borel probability measure $\nu$ is closely connected
with its \emph{upper }and \emph{lower pointwise dimension}
\begin{eqnarray*}
\overline{\dim}\;\nu(x):=\limsup_{\epsilon\to0}\frac{\log\nu\left(B_{\epsilon}(x)\right)}{\log\epsilon},\;\underline{\dim}\;
\nu(x):=\liminf_{\epsilon\to0}\frac{\log\nu\left(B_{\epsilon}(x)\right)}{\log\epsilon},
\end{eqnarray*}
where $B_{\epsilon}(x)$ denotes the closed ball of radius $\epsilon$
which is centered at $x$ (cf. \cite{Fal:97}). According to \cite[Propsition 3.3]{LG:92},
for $\mu$-a.e. $x$, the upper and lower pointwise dimension of $\mu$
at $x$ coincide and the common value equals $s_{0}$. Thus, by \cite[Corollary 2.1]{zhu:12},
$D_{0}(\mu)$ exists and equals $s_{0}$.

Also for the $L_{r}$-quantization with $r>0$, the second author
has given a characterization for the upper and lower quantization
dimension of $\mu$ in \cite{Zhu:09}. In some special cases, this
characterization leads to the existence of $D_{r}(\mu)$, and in these
situations its value also coincides with $s_{r}$ (cf. Corollary \ref{cor:Permutation}).}
\end{remark}
\begin{remark}
\label{remGL}{\rm Next, we recall some known results on self-similar measures.
For this let $(S_{i})_{i=1}^{N}$ be a set of contracting similitudes
on $\mathbb{R}^{q}$ with contraction ratios $(c_{i})_{i=1}^{N}$
and $(p_{i})_{i=1}^{N}$ a probability vector with $p_{i}>0$ for
all $1\leq i\leq N$. We denote the corresponding (unique compact
non-empty) self-similar set by $E=\bigcup_{i=1}^{N}S_{i}(E)$ and
the self-similar measure supported on $E$ by $\nu=\sum_{i=1}^{N}p_{i}\nu\circ S_{i}^{-1}.$
For $r\in[0,\infty)$, let $k_{r}$ be the positive real number given
by
\begin{eqnarray*}
k_{0}:=\bigg(\sum_{i=1}^{N}p_{i}\log c_{i}\bigg)^{-1}\sum_{i=1}^{N}p_{i}\log p_{i},\quad\sum_{i=1}^{N}(p_{i}c_{i}^{r})^{\frac{k_{r}}{k_{r}+r}}=1,\; r>0.\;\;
\end{eqnarray*}
Assume that $(S_{i})_{i=1}^{N}$ satisfies the open set condition,
namely, there is a non-empty bounded open set $U$ such that $S_{i}(U)\subset U$
for all $1\leq i\leq N$, and $S_{i}(U)\cap S_{j}(U)=\emptyset$ for
all $1\leq i\neq j\leq N$. Then Graf and Luschgy \cite{GL:01} proved
that $D_{r}(\nu)=k_{r}$ and that the $k_r$-dimensional upper and lower quantization coefficient for $\nu$ of order $r$ are both positive and finite.}
\end{remark}
\begin{remark}{\rm
Finally, let us point out that the strict affine situation
differs from the conformal situation significantly. In fact, for self-conformal
measures $\nu$ (as given e.g. in Remark \ref{remGL}) the quantization
dimension can be deduced from the multifractal formalism as follows,
see \cite{GL:01,LM:02}. If $T:\mathbb{R}\to\mathbb{R}$ denotes the
multifractal temperature function for the conformal system, then its Legendre transform $\widehat{T}$ determines the
multifractal $f\left(\alpha\right)$-spectrum of $\nu$, i.e.
\[
f\left(\alpha\right):=\dim_{H}\left(\left\{ x:\overline{\dim}\;\nu(x)=\underline{\dim}\;\nu(x)=\alpha\right\} \right)=\max\left\{ \widehat{T}\left(\alpha\right),0\right\} .
\]
On the other hand, for any $r\geq0$, there is a unique number
$\vartheta_{r}>0$ such that $T\left(\vartheta_{r}\right)=r\vartheta_{r}$
and we have $D_{r}\left(\nu\right)=T\left(\vartheta_{r}\right)/\left(1-\vartheta_{r}\right)$
(cf. \cite[Theorem 1]{LM:02}).

Also for Bedford-McMullen carpets the multifractal problem has been
solved -- see \cite{King:95,J:11}. In this situation the multifractal
temperature function $T$ is given implicitly by
\begin{eqnarray}
m^{-T(t)}\sum_{(i,j)\in G}p_{ij}^{t}q_{j}^{(1-\theta)t}\bigg(\sum_{h\in G_{x,j}}p_{hj}^{t}\bigg)^{\theta-1}=1.\label{k2}
\end{eqnarray}
It has been shown in \cite{King:95,J:11} that $T$ is a smooth convex
function and that its Legendre transform gives the multifractal spectrum
$f(\alpha)$ for the measure $\mu$ as above. As before, for $r\geq0$,
let $q_{r}$ be the unique number satisfying $T(\vartheta_{r})=r\vartheta_{r}$.
Setting $t=\vartheta_{r}$ in (\ref{k2}) we can rewrite this requirement
as
\[
m^{-r\vartheta_{r}}\sum_{(i,j)\in G}p_{ij}^{\vartheta_{r}}\bigg(q_{j}^{-\vartheta_{r}}\sum_{h\in G_{x,j}}p_{hj}^{\vartheta_{r}}\bigg)^{\theta-1}=1.
\]
So one can see that $T\left(\vartheta_{r}\right)/\left(1-\vartheta_{r}\right)$
coincides with the number $t_{r}$ given by
\begin{eqnarray}
m^{-r\frac{t_{r}}{t_{r}+r}}\sum_{j\in G_{y}}q_{j}^{(1-\theta)\frac{t_{r}}{t_{r}+r}}\bigg(\sum_{i\in G_{x,j}}p_{ij}^{\frac{t_{r}}{t_{r}+r}}\bigg)^{\theta} & = & 1.\label{maineq2}
\end{eqnarray}
Hence, by H\"{o}lder's inequality, $t_{r}\leq s_{r}$ and equality holds
if and only if $C_{j,r}$ coincide for all $j\in G_{y}$. Thus, by
Theorem \ref{mthm1}, this allows for the strict inequality
\[
T(\vartheta_{r})(1-\vartheta_{r})^{-1}=t_{r}<s_{r}=D_{r}(\mu).
\]

}\end{remark}

\section{Preliminaries}

As in \cite{Zhu:09}, to avoid degenerate cases, we always assume
\begin{equation}
m<n,\;{\rm card}\left(G_{x}\right),{\rm card}\left(G_{y}\right)\geq2.\label{hypo1}
\end{equation}
For $x\in\mathbb{R}$ let $[x]$ denote the largest integer not exceeding
$x$. For each $k\in\mathbb{N}$, set
\[
\ell(k):=[k\theta];\;\;\Omega_{k}:=G^{\ell(k)}\times G_{y}^{k-\ell(k)},\;\Omega^{*}:=\bigcup_{k\geq1}\Omega_{k}.
\]
Let $E_{0}:=[0,1]^{2}$. For $\sigma:=\left((i_{1},j_{1}),\ldots,(i_{k},j_{k})\right)\in G^{k}$.
Define
\begin{eqnarray*}
|\sigma|=k,\; E_{\sigma}:=f_{i_{1}j_{1}}\circ\ldots\circ f_{i_{k}j_{k}}(E_{0}),\;\;\mu_{\sigma}:=\mu\left(E_{\sigma}\right)=\prod_{h=1}^{k}p_{i_{h}j_{h}}.
\end{eqnarray*}
We call $E_{\sigma}$ a \emph{cylinder of order} $k$. For the above
$\sigma\in G^{k}$, we write
\[
\sigma^{-}:=((i_{1},j_{1}),\ldots,(i_{k-1},j_{k-1})).
\]
To each
\begin{equation}
\sigma=\big((i_{1},j_{1}),\ldots,(i_{\ell(k)},j_{\ell(k)}),j_{\ell(k)+1},\ldots,j_{k}\big)\in\Omega^{*},\label{sg1}
\end{equation}
there corresponds a unique rectangle, called an \emph{approximate
square of order} $k$:
\[
F_{\sigma}:=\bigg[\frac{p}{n^{\ell(k)}},\frac{p+1}{n^{\ell(k)}}\bigg]\times\bigg[\frac{q}{m^{k}},\frac{q+1}{m^{k}}\bigg],
\]
where $p:=\sum_{h=1}^{\ell(k)}i_{h}n^{\ell(k)-h}$, $q:=\sum_{h=1}^{k}j_{h}m^{k-h}$.
For $\sigma\in\Omega^{*}$ in (\ref{sg1}), we define
\begin{eqnarray*}
|\sigma|:=k,\;\;\mu_{\sigma}:=\mu\left(F_{\sigma}\right)=\prod_{h=1}^{\ell(k)}p_{i_{h}j_{h}}\prod_{h=\ell(k)}^{k}q_{j_{h}},\\
\sigma_{a}:=\left((i_{1},j_{1}),\ldots,(i_{\ell(k)},j_{\ell(k)})\right),\;\sigma_{b}:=\left(j_{\ell(k)+1},\ldots,j_{k}\right).
\end{eqnarray*}
and write $\sigma_{a}\prec\sigma$.
Let $|A|$ denote the diameter of a set $A\subset\mathbb{R}^{2}$.
One easily sees
\begin{equation}
m^{-|\sigma|}\leq|F_{\sigma}|\leq m^{-|\sigma|}\sqrt{n^{2}+1}=:\delta m^{-|\sigma|}\;\;{\rm with}\;\;\delta:=\sqrt{n^{2}+1}.\label{diameter}
\end{equation}
Let $\sigma,\tau\in\Omega^{*}$. We write $\sigma\prec\tau$ if $F_{\tau}\subset F_{\sigma}$;
and
\begin{eqnarray*}
\sigma=\tau^{\flat}\;\;{\rm if}\;\;\sigma\prec\tau\;\;{\rm and}\;\;|\tau|=|\sigma|+1.
\end{eqnarray*}
Thus, for the word $\sigma$ in (\ref{sg1}), $\sigma^{\flat}$ takes
the following two possible forms:
\[
\left\{ \begin{array}{ll}
\big((i_{1},j_{1}),\ldots,(i_{\ell(k)},j_{\ell(k)}),j_{\ell(k)+1},\ldots,j_{k-1}\big),\;\;\;\;\ell(k)=\ell(k-1),\\
\big((i_{1},j_{1}),\ldots,(i_{\ell(k)-1},j_{\ell(k)-1}),j_{\ell(k)},\ldots,j_{k-1}\big),\;\ell(k)=\ell(k-1)+1
\end{array}\right..
\]
We say that $\sigma,\tau\in\Omega^{*}$ are incomparable if neither
$\sigma\prec\tau$ nor $\tau\prec\sigma$. A finite set $\Gamma\subset\Omega^{*}$
is called a \emph{finite antichain} if any two words $\sigma,\tau\in\Gamma$
are incomparable; a finite antichain $\Gamma$ is called maximal if
$E\subset\bigcup_{\sigma\in\Gamma}F_{\sigma}$.

\subsection{On the $L_{r}$-quantization}

For $r>0$, we set $\underline{\eta}_{1,r}:=\min_{\sigma\in\Omega_{1}}\mu_{\sigma}m^{-r}$ and
\[
R_{1,r}:=\min_{(i,j)\in G}\min_{k\in G_{y}}p_{ij}q_{k}m^{-r},\;\;\underline{\eta}_{r}:=\min\left\{ R_{1,r},\underline{\eta}_{1,r}\right\} .
\]
We will need the finite maximal antichains as defined below:
\begin{equation}
\Gamma_{j,r}:=\{\sigma\in\Omega^{*}:\mu_{\sigma^{\flat}}m^{-|\sigma^{\flat}|r}\geq j^{-1}\underline{\eta}_{r}>\mu_{\sigma}m^{-|\sigma|r}\},\; j\in\mathbb{N}.\label{gammaj}
\end{equation}
Clearly, $\bigcup_{\sigma\in\Gamma_{j,r}}F_{\sigma}\supset E$ and
the interiors of $F_{\sigma},\sigma\in\Gamma_{j,r}$ are pairwise
disjoint. Set
\begin{eqnarray}
l_{1j}:=\min_{\sigma\in\Gamma_{j,r}}|\sigma|,\; l_{2j}:=\max_{\sigma\in\Gamma_{j,r}}|\sigma|;\;\;\overline{\eta}_{1,r}:=\max_{\sigma\in\Omega_{1}}\mu_{\sigma}m^{-r},\nonumber \\
R_{2,r}:=\max_{(i,j)\in G}p_{ij}q_{j}^{-1}m^{-r},\;\overline{\eta}_{r}:=\max\{\overline{\eta}_{1,r},\widetilde{R}_{2}\}.\label{kz8}
\end{eqnarray}

\begin{lemma}
There exist two constants $A_{1},A_{2}>0$ such that
\begin{equation}
A_{1}\log j\leq l_{1j}\leq l_{2j}\leq A_{2}\log j\;\;{\rm for\; large}\;\; j\in\mathbb{N}.\label{kj}
\end{equation}
 \end{lemma}
\begin{proof}
Note that there are two words $\sigma^{(i)}\in\Gamma_{j,r}\cap\Omega_{l_{ij}},i=1,2$.
By (\ref{gammaj}),
\[
\underline{\eta}_{r}^{l_{1j}}\leq\mu_{\sigma}m^{-|\sigma|r}<j^{-1}\underline{\eta}_{r}<j^{-1},\;\;
\overline{\eta}_{r}^{l_{2j}-1}\geq\mu_{\tau^{\flat}}m^{-|\tau^{\flat}|r}\geq j^{-1}\underline{\eta}_{r}.
\]
Hence, it suffices to set $A_{1}:=(-\log\underline{\eta}_{r})^{-1}$
and $A_{2}:=2(-\log\overline{\eta}_{r}^{-1})$.
\end{proof}
For every $j\in\mathbb{N}$, let $t_{j,r}$ be the unique positive
real number such that
\[
\sum_{\sigma\in\Gamma_{j,r}}(\mu_{\sigma}m^{-|\sigma|r})^{\frac{t_{j,r}}{t_{j,r}+r}}=1.
\]
Let $N_{j,r}:={\rm card}(\Gamma_{j,r})$. By (\ref{gammaj}) and
\cite[(3.2)]{Zhu:09}, we see
\begin{equation}
(j\underline{\eta}_{r}^{-1})^{\frac{t_{j,r}}{t_{j,r}+r}}\leq N_{j,r}\leq(jR_{1,r}^{-1}\underline{\eta}_{r}^{-1})^{\frac{t_{j,r}}{t_{j,r}+r}}\leq(j\underline{\eta}_{r}^{-2})^{\frac{t_{j,r}}{t_{j,r}+r}}.\label{cardgammaj}
\end{equation}
As is shown in the proof of \cite[Proposition 3.4]{Zhu:09}, we have
\begin{eqnarray}
\overline{D}_{r}(\mu)=\limsup_{j\to\infty}t_{j,r},\;\;\underline{D}_{r}(\mu)=\liminf_{j\to\infty}t_{j,r}.\label{kz2}
\end{eqnarray}
More exactly, there is a constant $D>0$, which is independent
of $j$, such that
\begin{eqnarray}
D\sum_{\sigma\in\Gamma_{j,r}}\mu_{\sigma}m^{-|\sigma|r}\leq e_{N_{j,r},r}^{r}(\mu)\leq\sum_{\sigma\in\Gamma_{j,r}}\mu_{\sigma}m^{-|\sigma|r}.\label{key2}
\end{eqnarray}

\begin{remark}{\rm
The first part of the proof of Proposition 3.4 of \cite{Zhu:09} is
to choose, for each $\sigma\in\Gamma_{j,r}$, a word $\widetilde{\sigma}\in\Omega^{*}$
such that $F_{\widetilde{\sigma}}\subset F_{\sigma}$ and
\begin{equation}
d(F_{\widetilde{\tau}},F_{\widetilde{\sigma}})\geq\beta\max\{|F_{\widetilde{\sigma}}|,|F_{\widetilde{\tau}}|\}.\label{add01}
\end{equation}
for some constant $\beta>0$ and every pair $\sigma,\tau\in\Gamma_{j,r}$
with $\sigma\neq\tau$. In fact, this can be seen by a more straightforward
argument. Let $H_{1},H_{2}\in\mathbb{N}$ satisfy
\[
\ell(k+H_{1})=\ell(k),\;\;\ell(k+H_{2})=\ell(k)+H_{2}.
\]
Then, by the definition of $\ell(k),k\in\mathbb{N}$, we have
\begin{eqnarray*}
(k+H_{1})\theta-1<k\theta,\;(k+H_{2})\theta\geq k\theta-1+H_{2}.
\end{eqnarray*}
Hence, $H_{1}<\theta^{-1}$ and $H_{2}\leq(1-\theta)^{-1}$. Let
\begin{equation}
\sigma=((i_{1},j_{1}),\ldots,(i_{l},j_{l}),j_{l+1},\ldots,j_{k})\in\Omega^{*}.\label{pre8}
\end{equation}
By (\ref{hypo1}), we can choose an approximate square $F_{\widetilde{\sigma}}$
of order $k+2(H_{1}+H_{2})$ such that $d(F_{\widetilde{\sigma}},F_{\sigma}^{c})\geq(n^{2}+1)^{-\frac{1}{2}}|F_{\widetilde{\sigma}}|$.
Hence, it suffices to set $\beta:=\delta^{-1}=(n^{2}+1)^{-\frac{1}{2}}$.}
\end{remark}

\subsection{On the geometric mean error}

For $k\in\mathbb{N}$, we simply write $C_{k}(\mu)$ for $C_{k,0}(\mu)$.
We will consider $\hat{e}_{k}(\mu):=\log e_{k,0}(\mu)$ instead of
$e_{k,0}(\mu)$ for convenience (cf. \cite{GL:04}). Set
\begin{eqnarray*}
\underline{q} & := & \min_{j\in G_{y}}q_{j};\;\;\underline{R}:=\min_{(i,j)\in G}p_{ij}\underline{q},\;\;\eta_{0}:=\min\{\underline{R},\min_{\sigma\in\Omega_{1}}\mu_{\sigma}\}.
\end{eqnarray*}
For every $j\geq1$ and $k\geq1$, we define
\begin{eqnarray}
\Lambda_{j} & := & \{\sigma\in\Omega^{*}:\mu_{\sigma^{\flat}}\geq j^{-1}\eta_{0}>\mu_{\sigma}\};\; k_{1j}:=\min_{\sigma\in\Lambda_{j}}|\sigma|,\; k_{2j}:=\max_{\sigma\in\Lambda_{j}}|\sigma|;\nonumber \\
t_{j} & := & \frac{\sum_{\sigma\in\Lambda_{j}}\mu_{\sigma}\log\mu_{\sigma}}{\sum_{\sigma\in\Lambda_{j}}\mu_{\sigma}\log m^{-|\sigma|}},\;\; s_{k,0}:=\frac{\sum_{\sigma\in\Omega_{k}}\mu_{\sigma}\log\mu_{\sigma}}{\sum_{\sigma\in\Omega_{k}}\mu_{\sigma}\log m^{-k}};\nonumber \\
Q_{j} & := & \psi_{j}^{\frac{1}{s_{0}}}e_{\psi_{j}}(\mu),\;\underline{Q}_{\flat}^{s_{0}}(\mu):=\liminf_{j\to\infty}Q_{j},\;\;\overline{Q}_{\flat}^{s_{0}}(\mu):=\limsup_{j\to\infty}Q_{j}.\label{kz9}
\end{eqnarray}
As we did in the proof for \cite[Lemma 4.1(b)]{Zhu:09}, it is not
difficult to show
\begin{lemma}
\label{rem1}With the above notations, we have

\begin{enumerate}[(A)]

\item $\left[j\eta_{0}^{-1}\right]\leq\psi_{j}:={\rm card}(\Lambda_{j})\leq\left[j\eta_{0}^{-2}\right]$;

\item  \label{lem:BulletB}there are constants $C_{i}>0,1\leq i\leq4$
such that
\begin{eqnarray*}
C_{1}\leq t_{j}\leq C_{2},\;\; C_{3}\log j\leq k_{1j}\leq k_{2j}\leq C_{4}\log j;
\end{eqnarray*}

\item  \label{lem:BulletC}$\underline{Q}_{0}^{s_{0}}(\mu)>0$ iff
$\underline{Q}_{\flat}^{s_{0}}(\mu)>0$ and $\overline{Q}_{0}^{s_{0}}(\mu)<\infty$
iff $\overline{Q}_{\flat}^{s_{0}}(\mu)<\infty$.

\end{enumerate}
\end{lemma}
For every $j\in G_{y}$, we define a contractive mapping $g_{j}$
by
\[
g_{j}(x,y):=(x,m^{-1}y)+(0,m^{-1}j),\;\;(x,y)\in\mathbb{R}^{2}.
\]
For $\sigma\in\Omega_{k}$ of the form (\ref{pre8}), we consider
a mapping $h_{\sigma}$ on $\mathbb{R}^2$:
\begin{eqnarray*}
h_{\sigma}(x)=f_{i_{1}j_{1}}\circ f_{i_{2}j_{2}}\ldots\circ f_{i_{\ell(k)}j_{\ell(k)}}\circ g_{j_{\ell(k)+1}}\ldots\circ g_{j_{k}}(x).
\end{eqnarray*}
Clearly, $h_{\sigma}$ is a Borel bijection satisfying $h_{\sigma}(E_{0})=F_{\sigma}$
and we have
\begin{eqnarray}
m^{-|\sigma|}d(x,y)\leq d\big(h_{\sigma}(x),h_{\sigma}(y)\big)\leq nm^{-|\sigma|}d(x,y),\;\; x,y\in\mathbb{R}^{2}.\label{pre7}
\end{eqnarray}
For every $\sigma\in\Omega^{*}$, we have a probability measure $\nu_{\sigma}:=\mu(\cdot|F_{\sigma})\circ h_{\sigma}$.
This measure is supported on $h_{\sigma}^{-1}(F_{\sigma}\cap E)\subset E_{0}$.
\begin{lemma}
 There exist constants $C>0,t>0$ such that, for all $\epsilon>0$,
\begin{equation}
\sup_{\sigma\in\Omega^{*}}\sup_{x\in\mathbb{R}^{2}}\nu_{\sigma}(B_{\epsilon}(x))\leq C\epsilon^{t}.\label{pre5}
\end{equation}
\end{lemma}
\begin{proof}
Let $k\geq1$ and $\sigma\in\Omega_{k}$ be fixed. We define
\begin{equation}
\Lambda(\sigma,h):=\{\tau\in\Omega_{k+h}:\sigma\prec\tau\},\;\; h\in\mathbb{N}.\label{san11}
\end{equation}
For $\tau,\rho\in\Lambda(\sigma,h)$, $|F_{\tau}|=|F_{\rho}|$ and
$|h_{\sigma}^{-1}(F_{\tau})|=|h_{\sigma}^{-1}(F_{\rho})|=:d_{h}(\sigma)$.
By (\ref{pre7}),
\begin{eqnarray}
n^{-1}m^{-h}\leq n^{-1}m^{k}|F_{\tau}|\leq d_{h}(\sigma)\leq m^{k}|F_{\tau}|\leq\delta m^{-h},\label{pre80}
\end{eqnarray}
where $\delta=\sqrt{n^{2}+1}$ as before. By (\ref{hyposep}),
for $x\in F_{\tau},y\in F_{\rho}$ with $\tau\neq\rho$, we have
\begin{eqnarray*}
d(x,y)\geq d(F_{\tau},F_{\rho})\geq\delta^{-1}|F_{\tau}|,
\end{eqnarray*}
This and (\ref{pre7}) implies that
\begin{equation}
d(h_{\sigma}^{-1}(F_{\tau}),h_{\sigma}^{-1}(F_{\rho}))\geq n^{-1}m^{k}\delta^{-1}|F_{\tau}|\geq n^{-1}\delta^{-1}|h_{\sigma}^{-1}(F_{\tau})|.
\end{equation}
As we see in \cite[Lemma 3.1]{Zhu:09}, for every $\rho\in\Omega_{k+1}$,
\begin{equation}
\frac{\mu_{\sigma}}{\mu_{\sigma^{\flat}}}=\left\{ \begin{array}{ll}
q_{j_{k+1}}\;\;\;\;\;\;\;\;\;\;\;\;\;\;\; & \mbox{if}\;\ell(k+1)=\ell(k)\\
p_{i_{\ell(k)+1}j_{\ell(k)+1}}q_{j_{k+1}}/q_{j_{\ell(k)+1}}\;\;\;\;\;\; & \mbox{if}\;\ell(k+1)=\ell(k)+1
\end{array}.\right.\label{measratio}
\end{equation}
Hence, for every $\tau\in\Lambda(\sigma,h)$ and $\overline{q}:=\max_{j\in G_{y}}q_{j}$,
we have
\begin{eqnarray}
\nu_{\sigma}(h_{\sigma}^{-1}(F_{\tau}))= & = & \mu_{\tau}\mu_{\sigma}^{-1}\leq\overline{q}^{h}.\label{pre90}
\end{eqnarray}
Let $t:=-\frac{\log\overline{q}}{\log m}$. Then by (\ref{pre80}),
(\ref{pre90}) and \cite[p. 737]{Hut:81} (cf. Lemma 3.1 of \cite{zhu:12}),
one sees that there is some constant $\widetilde{C}>0$ such that
\[
\nu_{\sigma}(B_{\epsilon}(x))\leq\widetilde{C}\epsilon^{t}\;\;{\rm for\; all}\;\;\epsilon\in(0,m^{-1})\;\;{\rm and}\;\; x\in\mathbb{R}^{2}.
\]
The proof of the lemma is then complete by \cite[Lemma 12.3]{GL:00}. \end{proof}
\begin{remark}
\label{rem2} {\rm Let $B_{k}:=t^{-1}(\log k+C)$. By (\ref{pre5}) and
the proof of Theorem 3.4 of \cite[p. 703]{GL:04}, one easily sees
that $\inf_{\sigma\in\Omega^{*}}\hat{e}_{k}(\nu_{\sigma})\geq B_{k}$.
Further, by (\ref{pre5}) and \cite[p.713]{GL:04}, for every pair
$p,q>1$ with $p^{-1}+q^{-1}=1$,
\begin{eqnarray*}
\hat{e}_{n}(\mu)-\hat{e}_{n+1}(\mu)\leq|\log(3|E|)|(n+1)^{-1}+C^{1/q}q t^{-1}(n+1)^{-1/p},
\end{eqnarray*}
Hence, for fixed integers $k_{1},k_{2},k_{3}\geq1$, one can find
an integer $L$ such that $k\geq L$ implies (cf. \cite[Lemma 2.2]{Zhu:13})
$\hat{e}_{k-k_{1}-k_{2}}(\nu_{\sigma})-\hat{e}_{k+k_{3}}(\nu_{\sigma})<\log2$
for all $\sigma\in\Omega^{*}$.}
\end{remark}
\begin{remark}
\label{rem3} {\rm For $\epsilon>0$ and a set $A\subset\mathbb{R}^{2}$,
let $(A)_{\epsilon}$ be the closed $\epsilon$-neighborhood of $A$.
Let $L_{1}$ be the smallest number of closed balls of radii $8^{-1}|F_{\sigma}|$
which are centered in $F_{\sigma}$ and cover $F_{\sigma}$, and let
$\gamma_{\sigma}$ be the set of centers of such $L_{1}$ balls. We define
\begin{eqnarray*}
\beta_{\sigma}(\alpha):=\alpha\cap(F_{\sigma})_{8^{-1}\delta^{-1}|F_{\sigma}|}
\;\;{\rm and}\;\; l_{\sigma}(\alpha):={\rm card}(\beta_{\sigma}(\alpha));\;\;\alpha\subset\mathbb{R}^{2}.
\end{eqnarray*}
Then, by the definition of $\nu_\sigma$, we deduce
\begin{eqnarray*}
\int_{F_{\sigma}}\log d(x,\gamma)d\mu(x) & \geq & \int_{F_{\sigma}}\log d(x,\beta_{\sigma}(\alpha)\cup\gamma_\sigma)d\mu(x)\\
 & = & \mu_{\sigma}\int\log d(x,\beta_{\sigma}\cup\gamma)d\nu_{\sigma}\circ h_{\sigma}^{-1}(x)\\
 & \geq & \mu_{\sigma}(\log m^{-|\sigma|}+\hat{e}_{l_{\sigma}(\gamma)+L_{1}}(\nu_{\sigma})).
\end{eqnarray*}}

\end{remark}
Next, we give an estimate $\hat{e}_{n}(\mu)$ for the subsequence
$(\psi_{j})_{j=1}^{\infty}$ of $(n)_{n=1}^{\infty}$.
\begin{lemma}
\label{pre2} There exist constants $C_{5},C_{6}$ such that
\begin{eqnarray}
\sum_{\sigma\in\Lambda_{j}}\mu_{\sigma}\log m^{-|\sigma|}+C_{5}\leq\hat{e}_{\psi_j}(\mu)\leq\sum_{\sigma\in\Lambda_{j}}\mu_{\sigma}\log m^{-|\sigma|}+C_{6}.\label{pre1}
\end{eqnarray}
\end{lemma}
\begin{proof}
For each $\sigma\in\Lambda_{j}$, we take an arbitrary point $a_{\sigma}\in F_{\sigma}$.
Then, using (\ref{diameter}),
\begin{eqnarray*}
\hat{e}_{\psi_j}(\mu) & \leq & \sum_{\sigma\in\Lambda_{j}}\int_{F_{\sigma}}\log d(x,a_{\sigma})d\mu(x)
\leq\sum_{\sigma\in\Lambda_{j}}\mu_{\sigma}\log\big(\delta m^{-|\sigma|}\big).
\end{eqnarray*}
Let $\sigma,\tau\in\Lambda_{j}$ with $|\sigma|\leq|\tau|$. By (\ref{hyposep}),
we have, $d(F_{\sigma},F_{\tau})\geq\delta^{-1}\max\{|F_{\sigma}|,|F_{\tau}|\}$.
Let $\alpha\in C_{\psi_{j}}(\mu)$. Using Remark \ref{rem2} and the
method in \cite[Proposition 3.4]{Zhu:10}, we can find a constant
$L\in\mathbb{N}$ such that $l_\sigma(\alpha)\leq L$
for all large $j$ and all $\sigma\in\Lambda_{j}$. Set $\overline{L}:=L+L_{1}$.
Then ${\rm card}(\beta_{\sigma}(\alpha)\cup\gamma_{\sigma})\leq\overline{L}$
for each $\sigma\in\Lambda_{j}$. By Remark \ref{rem3}, (\ref{pre7}),
the first part of Remark \ref{rem2} and \cite[Theorem 2.5]{GL:04}, we
deduce
\begin{eqnarray*}
\hat{e}_{\psi_j}(\mu) & \geq & \sum_{\sigma\in\Lambda_{j}}\mu_{\sigma}\int\log d(x,\beta_{\sigma}(\alpha^{^{}})\cup\gamma_{\sigma})d\nu_{\sigma}\circ h_{\sigma}^{-1}(x)\\
 & \geq & \sum_{\sigma\in\Lambda_{j}}\mu_{\sigma}(\log m^{-|\sigma|}+\hat{e}_{\overline{L}}(\nu_{\sigma}))\geq\sum_{\sigma\in\Lambda_{j}}\mu_{\sigma}\log m^{-|\sigma|}+B_{\overline{L}}.
\end{eqnarray*}
By setting $C_{5}:=B_{\overline{L}}$ and $C_{6}:=\log\delta$, the
lemma follows.
\end{proof}

\section{Proof of Theorem \ref{mthm1}}

For the proof of Theorem \ref{mthm1}, we need a series of lemmas.
For $r>0$, set
\[
P_{r}:=\sum_{(i,j)\in G}(p_{ij}m^{-r})^{\frac{s_{r}}{s_{r}+r}},\;\quad Q_{r}:=\sum_{j\in G_{y}}(q_{j}m^{-r})^{\frac{s_{r}}{s_{r}+r}}\,.
\]

\begin{lemma}
\label{kzlem1} For $r>0$, $s_{r}$ defined as in (\ref{maineq1})
and $\kappa_{r}:=s_{r}(s_{r}+r)^{-1}$ we have
\[
\kappa_{r}=\chi_{r}:=\inf\bigg\{ t\in\mathbb{R}:g(t):=\sum_{\sigma\in\Omega^{*}}(\mu_{\sigma}m^{-|\sigma|r})^{t}<\infty\bigg\}.
\]
\end{lemma}
\begin{proof}
By (\ref{maineq1}), we clearly have $Q_{r}\leq1\leq P_{r}$. For
every $k\geq1$, we have
\begin{eqnarray}
\sum_{\sigma\in\Omega_{k}}(\mu_{\sigma}m^{-|\sigma|r})^{\kappa_{r}}\geq P_{r}^{k\theta-1}Q_{r}^{k(1-\theta)+1}=P_{r}^{-1}Q_{r}>0.\label{kz7}
\end{eqnarray}
Hence, $g(\kappa_{r})=\infty$. Since $g$ is strictly decreasing,
we have, $\kappa_{r}\leq\chi_{r}$. On the other hand, by (\ref{maineq1}),
for any $t>\kappa_{r}$, we have,
\[
m^{-rt}\bigg(\sum_{(i,j)\in G}p_{ij}^{t}\bigg)^{\theta}\bigg(\sum_{j\in G_{y}}q_{j}^{t}\bigg)^{1-\theta}=:C(t)<1
\]
and hence
\[
g(t)=\sum_{k=1}^{\infty}\sum_{\sigma\in\Omega_{k}}(\mu_{\sigma}m^{-|\sigma|r})^{t}\leq\sum_{k=1}^{\infty}C(t)^{k}=\frac{C(t)}{1-C(t)}<\infty.
\]
This implies that $t\geq\chi_{r}$. By the arbitrariness of $t$,
we conclude that $\kappa_{r}\geq\chi_{r}$.
\end{proof}

Let $\underline{\eta}_{r}$ and $\overline{\eta}_{r}$ be as defined
in subsection 2.1. We set
\begin{eqnarray*}
\lambda_{1}:=-\log\underline{\eta}_{r}
\;\;{\rm and}\;\; \lambda_{2}:=-\log\overline{\eta}_{r}.
\end{eqnarray*}
Then, by (\ref{measratio}),
for every $\sigma\in\Omega^{*}$, we have
\begin{eqnarray}
e^{-\lambda_1} & \leq & \left(\mu_{\sigma}m^{-|\sigma|r}\right)\left(\mu_{\sigma^\flat}m^{-|\sigma^\flat|r}\right)^{-1}\leq e^{-\lambda_{2}(r)}.\label{kz6}
\end{eqnarray}
For $r>0$ and each $k\geq1$, we define
\begin{eqnarray}
\Lambda_{k,r} & := & \left\{ \sigma\in\Omega^{*}:e^{-(k+1)\lambda_1}\leq\mu_{\sigma}m^{-|\sigma|r}<e^{-k\lambda_1}\right\} ,\label{almostanti}\\
\widetilde{\Lambda}_{k,r} & := & \left\{ \sigma\in\Omega^{*}:\mu_{\sigma^{\flat}}m^{|\sigma^{\flat}|r}\geq e^{-k\lambda_1}>\mu_{\sigma}m^{-|\sigma|r}\right\} .\label{newanti}
\end{eqnarray}
We write $\varphi_{k,r}:={\rm card}(\Lambda_{k,r})$ and $\widetilde{\varphi}_{k,r}:={\rm card}(\widetilde{\Lambda}_{k,r})$.
Note that, $\Lambda_{k,r},k\geq1$, are pairwise disjoint; for every
$\sigma\in\Omega^{*}$, there is a unique $k\geq 0$ such that $\sigma\in\Lambda_{k,r}$.
Thus,
\begin{eqnarray}\label{kz10}
\Omega^{*}=\bigcup_{k=0}^{\infty}\Lambda_{k,r}\;\mbox{ and }\;\Lambda_{k_{1},r}\cap\Lambda_{k_{2},r}=\emptyset,\; k_{1}\neq k_{2}.\label{kz5}
\end{eqnarray}

\begin{lemma}
\label{kzlemma2} For every $r>0$ we have
\[
u_r:=\limsup_{k\to\infty}\frac{1}{\lambda_1k}\varphi_{k,r}=\kappa_{r}.
\]
\end{lemma}
\begin{proof}
The proof relies on the identity in Lemma \ref{kzlem1}. Fix $t>u_r$
and $0<\epsilon<t-u_r$. Then for all sufficiently large $k\in\mathbb{N}$
($k\geq n_{0},$ say) we have
\[
\frac{1}{\lambda_1k}\log\varphi_{k,r}\leq t-\epsilon.
\]
Therefore, with $R:=\sum_{k=0}^{n_{0}-1}\sum_{\sigma\in\Lambda_{k,r}}m^{-rt\left|\sigma\right|}\mu_{\sigma}^{t}$, by (\ref{kz10}),
we have
\begin{eqnarray*}
\sum_{\sigma\in\Omega^{*}}m^{-rt\left|\sigma\right|}\mu_{\sigma}^{t} & = & \sum_{k\in\mathbb{N}_{0}}\sum_{\sigma\in\Lambda_{k,r}}m^{-rt\left|\sigma\right|}\mu_{\sigma}^{t}\\
 & \leq & \sum_{k=0}^{n_{0}-1}\sum_{\sigma\in\Lambda_{k,r}}m^{-rt\left|\sigma\right|}\mu_{\sigma}^{t}+\sum_{k\geq n_{0}}\sum_{\sigma\in\Lambda_{k,r}}m^{-rt\left|\sigma\right|}\mu_{\sigma}^{t}\\
 & \leq & R+\sum_{k\geq n_{0}}\sum_{\sigma\in\Lambda_{k,r}}\mbox{e}^{-\lambda_1kt}=R+\sum_{k\geq n_{0}}\varphi_{k,r}\mbox{e}^{-\lambda_1kt}\\
 & \leq & R+\sum_{k\geq n_{0}}\mathrm{e}^{\left(t-\epsilon\right)\lambda_1k}\mathrm{e}^{-\lambda_1kt}=R+\sum_{k\geq n_{0}}\mathrm{e}^{-\epsilon\lambda_1k}<\infty.
\end{eqnarray*}
 Hence, by Lemma \ref{kzlem1}, $t\geq\kappa_{r}$. This shows that
$\kappa_{r}\leq u_r$.

Now suppose $t<u_r$ and fix $0<\epsilon<u_r-t$. Then there exists
a strictly increasing sequence $\left(n_{j}\right)\in\mathbb{N}^{\mathbb{N}}$
such that for all $j\in\mathbb{N}$ we have
\[
\frac{1}{\lambda_1n_{j}}\log\varphi_{n_{j},r}\geq t+\epsilon.
\]
This time, we therefore have
\begin{eqnarray*}
\sum_{\sigma\in\Omega^{*}}m^{-rt\left|\sigma\right|}\mu_{\sigma}^{t} & = & \sum_{n\in\mathbb{N}}\sum_{\sigma\in\Lambda_{k,r}}\mu_{\sigma}^{t}m^{-rt\left|\sigma\right|}\geq\sum_{j\in\mathbb{N}}
\sum_{\sigma\in\Lambda_{n_{j},r}}m^{-rt\left|\sigma\right|}\mu_{\sigma}^{t}\\
 & \geq & \sum_{j\in\mathbb{N}}\sum_{\sigma\in\Lambda_{n_{j},r}}\mathrm{e}^{-(n_j+1)\lambda_1t}=
 \sum_{j\in\mathbb{N}}\varphi_{n_{j},r}\mathrm{e}^{-\left(n_j+1\right)\lambda_1t}\\
 & \geq & \mathrm{e}^{-\lambda_1t}\sum_{j\in\mathbb{N}}\mathrm{e}^{\left(t+\epsilon\right)\lambda_1n_{j}}\mathrm{e}^{-\lambda_1 n_{j}t}=\mathrm{e}^{-\lambda_1t}\sum_{j\in\mathbb{N}}\mathrm{e}^{\epsilon\lambda_1n_{j}}=\infty.
\end{eqnarray*}
Hence, by Lemma \ref{kzlem1}, $t\leq\kappa_{r}$. It follows that $\kappa_{r}\geq u_r$. This completes the proof of the lemma.
\end{proof}
\begin{lemma}
\label{kzcommon}The sequences $\left(k^{-1}\log\varphi_{k,r}\right)_{k\in\mathbb{N}}$
and $\left(k^{-1}\log\widetilde{\varphi}_{k,r}\right)_{k\in\mathbb{N}}$
are both convergent and the corresponding limits coincide. \end{lemma}
\begin{proof}
We first show that, there is for some constant $M\in\mathbb{N}$,
such that
\begin{eqnarray}
\widetilde{\varphi}_{k,r}\leq\varphi_{k,r}\leq M\widetilde{\varphi}_{k,r}.\label{k1}
\end{eqnarray}
In fact, for every $\sigma\in\widetilde{\Lambda}_{k,r}$, by (\ref{kz6})
and (\ref{newanti}), we have
\begin{eqnarray*}
e^{-k\lambda_1}>\mu_{\sigma}m^{-|\sigma|r}\geq\mu_{\sigma^{\flat}}m^{|\sigma^{\flat}|r}\cdot e^{-\lambda_1}=e^{-(k+1)\lambda_1}.
\end{eqnarray*}
Hence, $\sigma\in\Lambda_{k,r}$ and $\widetilde{\Lambda}_{k,r}\subset\Lambda_{k,r}$.
It follows that $\widetilde{\varphi}_{k,r}\leq\varphi_{k,r}$. Next,
we show the inequality in the reverse direction. Let $\sigma$ be
an arbitrary word in $\Lambda_{k,r}$. As $\widetilde{\Lambda}_{k,r}$
is a finite maximal antichain, there is a word $\omega\in\widetilde{\Lambda}_{k,r}$
such that $\sigma\prec\omega$ or $\omega\prec\sigma$. However, if
$\sigma\prec\omega$, then we have
\begin{eqnarray*}
\mu_{\sigma}m^{-|\sigma|r}\geq\mu_{\omega^{\flat}}m^{-|\omega^{\flat}|r}\geq e^{-k\lambda_1}.
\end{eqnarray*}
This contradicts (\ref{almostanti}). Hence, $\omega\in\widetilde{\Lambda}_{k,r}$
and $\omega\prec\sigma$. Let $M_{0}$ be the smallest integer such
that $e^{M_{0}\lambda_2}\leq e^{-\lambda_1}$. Assume that
$|\sigma|-|\omega|>M_{0}$. Then by (\ref{kz6}),
\begin{eqnarray*}
\mu_{\sigma}m^{-|\sigma|r}\leq\mu_{\omega}m^{-|\omega|r}\cdot e^{-(M+1)\lambda_2}<e^{-k\lambda_1}\cdot e^{-\lambda_1}=e^{-(k+1)\lambda_1}.
\end{eqnarray*}
This again implies that $\sigma\notin\Lambda_{k,r}$, a contradiction.
By the above analysis, we conclude that $\Lambda_{k,r}\subset\bigcup_{\sigma\in\widetilde{\Lambda}_{k,r}}\Lambda(\sigma,M_{0})$.
Note that ${\rm card}(G_{x,j})\leq n$ and ${\rm card}(G_{y})\leq m$.
Hence,
\begin{eqnarray*}
{\rm card}(\Lambda(\sigma,M_{0}))\leq(mn)^{M_{0}}\;\;{\rm for\; all}\;\;\sigma\in\Omega^{*}.
\end{eqnarray*}
By setting $M:=(mn)^{M_{0}}+1$, (\ref{k1}) follows.

Now, it suffices to prove that the sequence $(k^{-1}\log\varphi_{k,r})_{k=1}^{\infty}$
is convergent. We complete the proof by showing that this sequence
is super-additive up to a constant difference. For this purpose, we
will establish a correspondence between elements of $\Lambda_{m+n,r}$
and those of $\Lambda_{m,r}$ and $\Lambda_{n,r}$.

Let $\sigma\in\Lambda_{m_{1},r}$ and $\omega\in\Lambda_{m_{2},r}$
be given. We write $k_{\sigma}:=|\sigma|,k_{\omega}:=|\omega|$ and
\begin{eqnarray*}
\sigma=((i_{1},j_{1}),\ldots,(i_{k_{\sigma}},j_{\ell(k_{\sigma}})),j_{\ell(k_{\sigma})+1},\ldots,j_{k_{\sigma}}),\\
\omega=((\widetilde{i}_{1},\widetilde{j}_{1}),\ldots,(\widetilde{i}_{\ell(k_{\omega})},\widetilde{j}_{\ell(k_{\omega})}),\widetilde{j}_{\ell(k_{\omega})+1},\ldots,\widetilde{j}_{k_{\omega}}).
\end{eqnarray*}
Then, by the definition of $\ell(k)$, we have
\begin{eqnarray*}
\ell(k_{\sigma})+\ell(k_{\omega})-1\leq\ell(k_{\sigma}+k_{\omega})=[(k_{\sigma}+k_{\omega})\theta]\leq\ell(k_{\sigma})+\ell(k_{\omega})+2.
\end{eqnarray*}
In the following, we need to distinguish two cases.
\begin{description}
\item [{Case~(1):}] $\ell(k_{\sigma}+k_{\omega})\geq\ell(k_{\sigma})+\ell(k_{\omega})$.
Let $H:=\ell(k_{\sigma}+k_{\omega})-\ell(k_{\sigma})-\ell(k_{\omega})$
and let $(\hat{i}_{h},\hat{j}_{h}),1\leq h\leq H$, be $H$ arbitrary
elements of $G$ and define
\begin{eqnarray*}
\rho=\rho(\sigma,\omega):=(\sigma_{a},\omega_{a},(\hat{i}_{1},\hat{j}_{1}),\ldots,(\hat{i}_{H},\hat{j}_{H}),\sigma_{b},\omega_{b})\in\Omega^{*}.
\end{eqnarray*}
Let $\underline{p}:=\min_{(i,j)\in G}p_{ij}$ and $B_{0}:=\underline{p}^{2}m^{-2r}$.
Note that $0\leq H\leq2$. We deduce
\begin{eqnarray*}
\mu_{\rho}m^{-|\rho|r}=\mu_{\sigma}m^{-|\sigma|r}\mu_{\omega}m^{-|\omega|r}\prod_{h=1}^Hp_{i_{h}j_{h}}m^{-r}\left\{ \begin{array}{ll}
<e^{-(m_{1}+m_{2})\lambda_1}\\
\geq B_{0}e^{-(m_{1}+m_{2}+2)\lambda_1}
\end{array}\right..
\end{eqnarray*}
For $\sigma^{(i)}\in\Lambda_{m,r}$, $\omega^{(i)}\in\Lambda_{n,r}$
with $\ell(k_{\sigma^{(i)}}+k_{\omega^{(i)}})\geq\ell(k_{\sigma^{(i)}})+\ell(k_{\omega^{(i)}})$
$i=1,2$, one can see the following equivalence:
\begin{eqnarray*}
\rho(\sigma^{(1)},\omega^{(1)})=\rho(\sigma^{(2)},\omega^{(2)})\;\;\mbox{iff}\;\;\sigma^{(1)}=\sigma^{(2)}\;\mbox{and }\;\omega^{(1)}=\omega^{(2)}.
\end{eqnarray*}

\item [{Case~(2):}] $\ell(k_{\sigma}+k_{\omega})=\ell(k_{\sigma})+\ell(k_{\omega})-1$.
In this case, we define
\begin{eqnarray*}
\rho(\sigma,\omega):=(\sigma_{a},\omega_{a}^{-},\sigma_{b},\omega_{b})\in\Omega^{*}.
\end{eqnarray*}
Let $B_{1}:=\underline{p}^{-1}m^{r}$. Since $\sigma\in\Lambda_{m_{1},r}$
and $\omega\in\Lambda_{m_{2},r}$, we have
\begin{eqnarray*}
\mu_{\rho}m^{-|\rho|r} & = & \frac{\mu_{\sigma}m^{-|\sigma|r}\mu_{\omega}m^{-|\omega|r}}{(p_{\widetilde{i}_{\ell(k_{\omega})}\widetilde{j}_{\ell(k_{\omega})}}m^{-r})}\left\{ \begin{array}{ll}
<B_{1}\mbox{e}^{-(m_{1}+m_{2})\lambda_1}\\
\geq\mbox{e}^{-(m_{1}+m_{2}+2)\lambda_1}
\end{array}\right..
\end{eqnarray*}
Let $\sigma^{(i)}\in\Lambda_{m_{1},r},\omega^{(i)}\in\Lambda_{m_{2},r}$
with $\ell(k_{\sigma^{(i)}}+k_{\omega^{(i)}})=\ell(k_{\sigma^{(i)}})+\ell(k_{\omega^{(i)}})-1,i=1,2$.
Then we have the following equivalence:
\begin{eqnarray*}
\rho(\sigma^{(1)},\omega^{(1)})=\rho(\sigma^{(2)},\omega^{(2)})\;\;\mbox{iff}\;\;\left\{ \begin{array}{ll}
\sigma^{(1)}=\sigma^{(2)}\\
(\omega_{a}^{(1)})^{-}=(\omega_{a}^{(2)})^{-};\;\sigma_{b}=\omega_{b}
\end{array}\right..
\end{eqnarray*}

\end{description}
We need to consider the following subset of $\Omega^{*}$:
\[
\Lambda_{m_{1},m_{2},r}:=\bigg\{\rho\in\Omega^{*}:B_{0}e^{-2\lambda_1}\leq\mbox{e}^{(m_{1}+m_{2})\lambda_1}\mu_{\rho}m^{-|\rho|r}<B_{1}\bigg\}.
\]
As we did for (\ref{k1}), one can show that, for some constants $B_{2},B_{3}>0$,
we have
\[
B_{2}\phi_{m_{1}+m_{2},r}\leq{\rm card}(\Lambda_{m_{1},m_{2},r})\leq B_{3}\phi_{m_{1}+m_{2},r}.
\]
Now combining case (1) and (2), one sees that, $\rho(\sigma,\omega)\in\Lambda_{m_{1},m_{2},r}$
for any pair $\sigma\in\Lambda_{m_{1},r},\omega\in\Lambda_{m_{2},r}$.
Moreover, we have
\[
\varphi_{m_{1},r}\cdot N^{-1}\varphi_{m_{2},r}\leq{\rm card}(\Lambda_{m_{1},m_{2},r})\leq B_{3}\phi_{m_{1}+m_{2},r}.
\]
By taking logarithms, it follows immediately that
\[
\log\varphi_{m_{1},r}+\log\varphi_{m_{2},r}-\log(NB_{3})\leq\log\phi_{m_{1}+m_{2},r}.
\]
which implies that $\lim_{k\to\infty}k^{-1}\log\varphi_{k,r}$ exists.
\end{proof}

\subsection{Proof of Theorem \ref{mthm1} }

We first show that $D_{r}(\mu)$ exists. For this purpose, we consider
the finite maximal antichains $\widetilde{\Lambda}_{k,r}$ as defined
in (\ref{newanti}). For $r>0$ let $\delta_{k,r}$ the unique solution
of
\[
\sum_{\sigma\in\widetilde{\Lambda}_{k,r}}(\mu_{\sigma}m^{-|\sigma|r})^{\frac{\delta_{k,r}}{\delta_{k,r}+r}}=1
\]
and set
\[
\overline{\delta}_{r}:=\limsup_{k\to\infty}\delta_{k,r},\;\underline{\delta}_{r}:=\liminf_{k\to\infty}\delta_{k,r}.
\]
Following the lines of \cite{zhu:12}, one can replace $\Gamma_{j}$
in there with $\widetilde{\Lambda}_{k,r}$ and obtain
\[
\overline{D}_{r}(\mu)=\overline{\delta}_{r},\;\;\underline{D}_{r}(\mu)=\underline{\delta}_{r}.
\]
Now by the definitions of $\widetilde{\Lambda}_{k,r}$ and $\delta_{k,r}$,
one gets
\begin{eqnarray*}
\widetilde{\varphi}_{k,r}\cdot e^{-k\lambda_1\frac{\delta_{k,r}}{\delta_{k,r}+r}}\geq1,\;\;\widetilde{\varphi}_{k,r}\cdot e^{-(k+1)\lambda_1\frac{\delta_{k,r}}{\delta_{k,r}+r}}\leq1.
\end{eqnarray*}
Taking logarithms on both sides of the preceding inequalities, we
have
\begin{eqnarray*}
\frac{1}{(k+1)\lambda_1}\log\widetilde{\phi}_{k,r}\leq\frac{\delta_{k,r}}{\delta_{k,r}+r}\leq\frac{1}{k\lambda_1}\log\widetilde{\phi}_{k,r}.
\end{eqnarray*}
This, combined with Lemma \ref{kzcommon} and Lemma \ref{kzlemma2},
yields
\begin{eqnarray*}
\lim_{k\to\infty}\frac{\delta_{k,r}}{\delta_{k,r}+r}=\lim_{k\to\infty}\frac{1}{k\lambda_1}\log\widetilde{\phi}_{k,r}=\lim_{k\to\infty}\frac{1}{k\lambda_1}\log\phi_{k,r}=\kappa_{r}=\frac{s_{r}}{s_{r}+r}.
\end{eqnarray*}
It hence follows that $D_{r}(\mu)=s_{r}$.

Finally, we will treat the remaining parts of the theorem separately.

\emph{ad (\ref{ConditionA}):} Assume that $C_{j,r},j\in G_{y}$ are
constant. We denote the common value by $\pi_{r}$. In this case,
$s_{r}=t_{r}$. Thus, by (\ref{maineq1}) or (\ref{maineq2}), we
have
\begin{equation}
P_{r}^{-1}\pi_{r}^{1-\theta}=1.\label{eq2}
\end{equation}
In order to show (\ref{main1}), we need an auxiliary probability
measure. Define
\[
\widetilde{p}_{ij}:=P_{r}^{-1}(p_{ij}m^{-r})^{\frac{s_{r}}{s_{r}+r}},\;(i,j)\in G.
\]
Let $\nu_{2}$ denote the self-affine measure on $E$ associated with
$(\widetilde{p}_{ij})_{(i,j)\in G}$. We have
\begin{eqnarray*}
\widetilde{q}_{j}:=\sum_{i\in G_{x,j}}\widetilde{p}_{ij}=P_{r}^{-1}\sum_{i\in G_{x,j}}(p_{ij}m^{-r})^{\frac{s_{r}}{s_{r}+r}},\; j\in G_{Y}.
\end{eqnarray*}
For $\sigma=((i_{1},j_{1}),\ldots,(i_{l},j_{l}),j_{l+1},\ldots,j_{k})\in\Omega^{*}$,
we have
\begin{eqnarray}
\nu_{2}(F_{\sigma}) & = & \prod_{h=1}^{l}P_{r}^{-1}(p_{i_{h}j_{h}}m^{-r})^{\frac{s_{r}}{s_{r}+r}}\prod_{h=l+1}^{k}P_{r}^{-1}\sum_{i\in G_{x,j_{h}}}(p_{ij_{h}}m^{-r})^{\frac{s_{r}}{s_{r}+r}}\nonumber \\
 & = & P_{r}^{-k}\prod_{h=1}^{l}(p_{i_{h}j_{h}}m^{-r})^{\frac{s_{r}}{s_{r}+r}}\prod_{h=l+1}^{k}(q_{j_{h}}m^{-r})^{\frac{s_{r}}{s_{r}+r}}\prod_{h=l+1}^{k}\sum_{i\in G_{x,j_{h}}}\bigg(\frac{p_{ij_{h}}}{q_{j_{h}}}\bigg)^{\frac{s_{r}}{s_{r}+r}}\nonumber \\
 & = & (\mu_{\sigma}m^{-|\sigma|r})^{\frac{s_{r}}{s_{r}+r}}P_{r}^{-k}\prod_{h=l+1}^{k}\sum_{i\in G_{x,j_{h}}}\bigg(\frac{p_{ij_{h}}}{q_{j_{h}}}\bigg)^{\frac{s_{r}}{s_{r}+r}}\nonumber \\
 & = & (\mu_{\sigma}m^{-|\sigma|r})^{\frac{s_{r}}{s_{r}+r}}P_{r}^{-k}\pi_{r}^{k-\ell(k)}.\label{inequal01}
\end{eqnarray}
Note that $\pi_{r}\geq1$. In view of (\ref{eq2}), we have
\[
1=(P_{r}^{-1}\pi_{r}^{1-\theta})^{k}\leq P_{r}^{-k}\pi_{r}^{k-\ell(k)}\leq(P_{r}^{-1}\pi_{r}^{1-\theta})^{k}\pi_{r}=\pi_{r}.
\]
This, together with (\ref{inequal01}), implies
\begin{eqnarray*}
\pi_{r}^{-1}=\pi_{r}^{-1}\sum_{\sigma\in\Gamma_{j,r}}\nu(F_{\sigma})\leq\sum_{\sigma\in\Gamma_{j,r}}(\mu_{\sigma}m^{-|\sigma|r})^{\frac{s_{r}}{s_{r}+r}}\leq\sum_{\sigma\in\Gamma_{j,r}}\nu_{2}(F_{\sigma})=1.
\end{eqnarray*}
By the definition of $\Gamma_{j,r}$, one gets
\begin{equation}
\pi_{r}^{-1}(j\underline{\eta}_{r}^{-1})^{\frac{s_{r}}{s_{r}+r}}\leq N_{j,r}\leq(j\underline{\eta}_{r}^{-2})^{\frac{s_{r}}{s_{r}+r}}.\label{cardestimate}
\end{equation}
Using this and (\ref{key2}), we deduce
\begin{eqnarray}
\xi_{j,r}: & = & N_{j,r}^{\frac{r}{s_{r}}}e_{N_{j,r},r}^{r}(\mu)\leq N_{j,r}^{\frac{r}{s_{r}}}\sum_{\sigma\in\Gamma_{j,r}}(\mu_{\sigma}m^{-|\sigma|r})\nonumber \\
 & \leq & N_{j,r}^{\frac{r}{s_{r}}}N_{j,r}\cdot(j^{-1}\underline{\eta}_{r})\leq\underline{\eta}_{r}^{-1}.\label{san6}
\end{eqnarray}
In a similar manner, by (\ref{key2}) and (\ref{cardestimate}), we
have
\begin{eqnarray}
\xi_{j,r} & \geq & DN_{j,r}^{\frac{r}{s_{r}}}\sum_{\sigma\in\Gamma_{j,r}}(\mu_{\sigma}m^{-|\sigma|r})\nonumber \\
 & \geq & DN_{j,r}^{\frac{r}{s_{r}}}N_{j,r}\cdot(j^{-1}\underline{\eta}_{r}^{2})\geq D\underline{\eta}_{r}\pi_{r}^{-(1+\frac{r}{s_{r}})}.\label{san6'}
\end{eqnarray}
Let $\overline{\eta}_{r}$ be as defined in (\ref{kz8}) and let $\Lambda(\sigma,h)$
be as defined in (\ref{san11}). Then, for all $j\geq(1-\overline{\eta}_{r})^{-1}-1=j_{0}$
and every $\sigma\in\Gamma_{j,r}$, we have,
\begin{eqnarray*}
\mu_{\omega}m^{-|\omega|r}\leq\overline{\eta}_{r}\mu_{\sigma}m^{-|\sigma|r}\leq(j+1)^{-1}\underline{\eta}_{r}\;\;{\rm for\; all}\;\;\omega\in\Lambda(\sigma,1)
\end{eqnarray*}
It follows that $N_{j,r}\leq N_{j+1,r}\leq(mn)N_{j,r}$. For
every $k\geq N_{j_{0},r}$, there is some $j\geq j_{0}$ such that
$N_{j,r}\leq k\leq N_{j+1,r}\leq(mn)N_{j,r}$. Hence, by Theorem
4.12 in \cite{GL:00}, we deduce
\begin{eqnarray*}
(mn)^{-\frac{1}{s_{r}}}N_{j+1,r}^{\frac{1}{s_{r}}}e_{\phi_{j+1},r}(\mu)\leq k^{\frac{1}{s_{r}}}e_{k,r}(\mu)\leq(mn)^{\frac{1}{s_{r}}}N_{j,r}^{\frac{1}{s_{r}}}e_{N_{j,r},r}(\mu).
\end{eqnarray*}
This, together with (\ref{san6}) and (\ref{san6'}), implies (\ref{main1}).
Let us remark that Theorem \ref{mthm1} (\ref{ConditionA}) improves
the result of \cite[Theorem 4.3]{Zhu:09}, where $(p_{ij})_{i\in G_{x,j}},j\in G_{y}$
are required to be permutations of one another.

\emph{ad (\ref{ConditionB}):} For $k\geq2$, we write $I_{k}:=\sum_{\omega\in\Omega_{k}}\mu_{\omega}\log\mu_{\omega}$.
Note that
\begin{eqnarray*}
\mu_{\omega}=\prod_{h=1}^{\ell(k)}p_{i_{h}j_{h}}\prod_{h=\ell(k)+1}^{k}q_{j_{h}}.
\end{eqnarray*}
for $\omega=((i_{1},j_{1}),\ldots,(i_{\ell(k)}j_{\ell(k)}),j_{\ell(k)+1},\ldots,j_{k})$.
We have
\begin{eqnarray*}
I_{k} & = & \ell(k)\sum_{(i,j)\in G}p_{ij}\log p_{ij}+(k-\ell(k))\sum_{j\in G_{y}}q_{j}\log q_{j},\\
I_{k+1} & = & \ell(k+1)\sum_{(i,j)\in G}p_{ij}\log p_{ij}+(k+1-\ell(k+1))\sum_{j\in G_{y}}q_{j}\log q_{j}.
\end{eqnarray*}
Hence, it follows that
\begin{equation}
I_{k+1}-I_{k}=\left\{ \begin{array}{ll}
\sum_{j\in G_{y}}q_{j}\log q_{j}\;\;\;\;\;\;\;\;\;\;\;{\rm if}\;\ell(k+1)=\ell(k)\\
\sum_{(i,j)\in G}p_{ij}\log p_{ij}\;\;\;\;\;\;{\rm if}\;\ell(k+1)=\ell(k)+1
\end{array}.\right.\label{ik}
\end{equation}
For $h\in\mathbb{N}$ and $\sigma\in\Omega_{k}$ with
\[
\sigma=((i_{1},j_{1}),\ldots,(i_{\ell(k)},j_{\ell(k)}),j_{\ell(k)+1},\ldots,j_{k}),
\]
let $\Lambda(\sigma,h)$ be as defined in (\ref{san11}). Next, with
the assumption in (\ref{ConditionB}),
we show
\begin{eqnarray}
\sum_{\omega\in\Lambda(\sigma,h)}\mu_{\omega}\log\mu_{\omega}=\mu_{\sigma}\log\mu_{\sigma}+\mu_{\sigma}(I_{k+h}-I_{k}).\label{san0}
\end{eqnarray}
First we show (\ref{san0}) for $h=1$. Note that $\sum_{\omega\in\Lambda(\sigma,1)}\mu_{\omega}=\mu_{\sigma}$.
We write
\begin{eqnarray*}
c(\sigma,1):=\sum_{\omega\in\Lambda(\sigma,1)}\mu_{\omega}\log\mu_{\omega}-\mu_{\sigma}\log\mu_{\sigma}=\mu_{\sigma}\sum_{\omega\in\Lambda(\sigma,1)}\frac{\mu_{\omega}}{\mu_{\sigma}}\log\frac{\mu_{\omega}}{\mu_{\sigma}}.
\end{eqnarray*}
If $\ell(k+1)=\ell(k)$, then, by (\ref{ik}) and (\ref{measratio}),
we have
\begin{equation}
c(\sigma,1)=\mu_{\sigma}\sum_{j\in G_{y}}q_{j}\log q_{j}=\mu_{\sigma}(I_{k+1}-I_{k}).\label{san1}
\end{equation}
If $\ell(k+1)=\ell(k)+1$, by (\ref{measratio}), we deduce
\begin{eqnarray*}
c(\sigma,1) & = & \mu_{\sigma}\sum_{i\in G_{x,j_{l_{k}+1}},j_{k+1}\in G_{y}}\frac{p_{ij_{\ell(k)+1}}q_{j_{k+1}}}{q_{j_{\ell(k)+1}}}\log\bigg(\frac{p_{ij_{\ell(k)+1}}q_{j_{k+1}}}{q_{j_{\ell(k)+1}}}\bigg)\\
 & = & \mu_{\sigma}\sum_{i\in G_{x,j_{l_{k}+1}},j_{k+1}\in G_{y}}\frac{p_{ij_{\ell(k)+1}}q_{j_{k+1}}}{q_{j_{\ell(k)+1}}}\log\frac{p_{ij_{\ell(k)+1}}}{q_{j_{\ell(k)+1}}}+\mu_{\sigma}\sum_{j\in G_{y}}q_{j}\log q_{j}\\
 & = & \mu_{\sigma}\sum_{i\in G_{x,j_{\ell(k)+1}}}\frac{p_{ij_{\ell(k)+1}}}{q_{j_{\ell(k)+1}}}\log\frac{p_{ij_{\ell(k)+1}}}{q_{j_{\ell(k)+1}}}+\mu_{\sigma}\sum_{j\in G_{y}}q_{j}\log q_{j}.
\end{eqnarray*}
By the hypothesis, $C_{j},j\in G_{y}$ are constant. Thus, in view
of (\ref{ik}), we have
\begin{eqnarray}
c(\sigma,1) & = & \mu_{\sigma}\sum_{j\in G_{y}}q_{j}\sum_{i\in G_{x,j}}\frac{p_{ij}}{q_{j}}\log\frac{p_{ij}}{q_{j}}+\mu_{\sigma}\sum_{j\in G_{y}}\sum_{i\in G_{x,j}}p_{ij}\log q_{j}\nonumber \\
 & = & \mu_{\sigma}\sum_{j\in G_{y}}\sum_{i\in G_{x,j}}p_{ij}\log\frac{p_{ij}}{q_{j}}+\mu_{\sigma}\sum_{j\in G_{y}}\sum_{i\in G_{x,j}}p_{ij}\log q_{j}\nonumber \\
 & = & \mu_{\sigma}\sum_{j\in G_{y}}\sum_{i\in G_{x,j}}p_{ij}\log p_{ij}=\mu_{\sigma}\sum_{(i,j)\in G}p_{ij}\log p_{ij}\nonumber \\
 & = & \mu_{\sigma}(I_{k+1}-I_{k}).\label{san2}
\end{eqnarray}
Combining (\ref{san1}) and (\ref{san2}), we conclude that, for $\sigma\in\Omega_{k}$,
\begin{eqnarray}
\sum_{\omega\in\Lambda(\sigma,1)}\mu_{\omega}\log\mu_{\omega} & = & \mu_{\sigma}\log\mu_{\sigma}+\mu_{\sigma}(I_{k+1}-I_{k}).\label{san3}
\end{eqnarray}
Assume that (\ref{san0}) holds for $h=p\in\mathbb{N}$. Next, we
show that it is true for $h=p+1$. Note that $\sum_{\omega\in\Lambda(\sigma,p)}\mu_{\omega}=\mu_{\sigma}$.
By (\ref{san3}), we deduce
\begin{eqnarray*}
\sum_{\tau\in\Lambda(\sigma,p+1)}\mu_{\tau}\log\mu_{\tau} & = & \sum_{\omega\in\Lambda(\sigma,p)}\sum_{\tau\in\Lambda(\omega,1)}\mu_{\tau}\log\mu_{\tau}\\
 & = & \sum_{\omega\in\Lambda(\sigma,p)}\big(\mu_{\omega}\log\mu_{\omega}+\mu_{\omega}(I_{k+p+1}-I_{k+p})\big)\\
 & = & \mu_{\sigma}\log\mu_{\sigma}+\mu_{\sigma}(I_{k+p}-I_{k})+\sum_{\omega\in\Lambda(\sigma,p)}\mu_{\omega}(I_{k+p+1}-I_{k+p})\\
 & = & \mu_{\sigma}\log\mu_{\sigma}+\mu_{\sigma}(I_{k+p+1}-I_{k}).
\end{eqnarray*}
Hence, by induction, (\ref{san0}) holds for all $h\in\mathbb{N}$.
Equivalently,
\[
\mu_{\sigma}(\log\mu_{\sigma}-I_{k})=\sum_{\omega\in\Lambda(\sigma,h)}\mu_{\omega}(\log\mu_{\omega}-I_{k+h}).
\]
In particular, for $h=k_{2j}-k$, we have
\[
\mu_{\sigma}(\log\mu_{\sigma}-I_{k})=\sum_{\omega\in\Lambda(\sigma,k_{2j}-k)}\mu_{\omega}(\log\mu_{\omega}-I_{k_{2j}}).
\]
By applying the preceding equation to all words $\sigma\in\Lambda_{j}$,
we obtain
\begin{eqnarray*}
\sum_{\sigma\in\Lambda_{j}}\mu_{\sigma}(\log\mu_{\sigma}-I_{|\sigma|}) & = & \sum_{\sigma\in\Lambda_{j}}\sum_{\omega\in\Lambda(\sigma,k_{2j}-k)}\mu_{\omega}(\log\mu_{\omega}-I_{k_{2j}})\\
 & = & \sum_{\omega\in\Omega_{k_{2j}}}\mu_{\omega}(\log\mu_{\omega}-I_{k_{2j}})=0.
\end{eqnarray*}
This implies that $\sum_{\sigma\in\Lambda_{j}}\mu_{\sigma}\log\mu_{\sigma}=\sum_{\sigma\in\Lambda_{j}}\mu_{\sigma}I_{|\sigma|}$.
On the other hand,
\begin{eqnarray*}
J_{k}:=\sum_{\sigma\in\Omega_{k}}\mu_{\sigma}\log m^{-k}=\log m^{-k},\;\sum_{\sigma\in\Lambda_{j}}\mu_{\sigma}\log m^{-|\sigma|}=\sum_{\sigma\in\Lambda_{j}}\mu_{\sigma}J_{|\sigma|}.
\end{eqnarray*}
As in \cite[Lemma 2.6]{Zhu:13}, there exist some integers $k_{j}^{(1)},k_{j}^{(2)}\in[k_{1j},k_{2j}]$
such that
\begin{eqnarray}
s_{k_{j}^{(1)},0}\leq t_{j} & = & \frac{\sum_{\sigma\in\Lambda_{j}}\mu_{\sigma}\log\mu_{\sigma}}{\sum_{\sigma\in\Lambda_{j}}\mu_{\sigma}\log m^{-|\sigma|}}=\frac{\sum_{\sigma\in\Lambda_{j}}\mu_{\sigma}I_{\sigma}}{\sum_{\sigma\in\Lambda_{j}}\mu_{\sigma}J_{\sigma}}\leq s_{k_{j}^{(2)},0}.\label{san5}
\end{eqnarray}
Let $Q_{j}$ and $\psi_{j}$ be as defined in (\ref{kz9}). By (\ref{san5}),
Lemmas \ref{pre2}, \ref{rem1}, we deduce
\begin{eqnarray}\label{pre3}
Q_{j} & \leq & s_{0}^{-1}\log\psi_{j}+\sum_{\sigma\in\Lambda_{j}}\mu_{\sigma}\log m^{-|\sigma|}+C_{5}\nonumber\\
 & = & s_{0}^{-1}\log\psi_{j}+t_{j}^{-1}\sum_{\sigma\in\Lambda_{j}}\mu_{\sigma}\log\mu_{\sigma}+C_{5}\nonumber \\
 & \leq & s_{0}^{-1}\log(j\eta_{0}^{-2})+t_{j}^{-1}\log(j^{-1}\eta_{0})+C_{5}\nonumber\\
 & \leq & (s_{0}^{-1}-t_{j}^{-1})\log j+C_{7},
\end{eqnarray}
where $C_{7}:=s_{0}^{-1}\log\eta_{0}^{-2}+C_{2}^{-1}\log\eta_{0}+C_{5}$.
On the other hand, we have
\begin{eqnarray*}
s_{k,0} & = & \frac{\ell(k)\sum_{(i,j)\in G}p_{ij}\log p_{ij}+(k-\ell(k)\sum_{j\in G_{y}}q_{j}\log q_{j}}{-k\log m}\\
 & = & \frac{k^{-1}\ell(k)\sum_{(i,j)\in G}p_{ij}\log p_{ij}+(1-k^{-1}\ell(k))\sum_{j\in G_{y}}q_{j}\log q_{j}}{-\log m}.
\end{eqnarray*}
Set $\chi:=(\log m)^{-1}\big(\sum_{(i,j)\in G}|p_{ij}\log p_{ij}|+\sum_{j\in G_{y}}|q_{j}\log q_{j}|\big)$.
Then
\begin{eqnarray}
\chi>0,\;\;|s_{k,0}-s_{0}|\leq k^{-1}\chi,\;\;{\rm implying}\;\;|s_{k,0}^{-1}-s_{0}^{-1}|\leq2\chi s_{0}^{-2}k^{-1}\label{pre4}
\end{eqnarray}
for all $k\geq2\chi s_{0}^{-1}$. Using (\ref{pre3}), (\ref{pre4})
and Lemma \ref{rem1} (\ref{lem:BulletB}), we deduce
\begin{eqnarray*}
Q_{j} & \leq & (s_{0}^{-1}-s_{k_{j}^{(1)},0}^{-1})\log j+C_{7}\\
 & \leq & 2\chi s_{0}^{-2}(k_{j}^{(1)})^{-1}\log j+C_{7}<2\chi s_{0}^{-2}C_{3}^{-1}+C_{7}.
\end{eqnarray*}
Hence, $\overline{Q}_{\flat}^{s_{0}}(\mu)<\infty$. By Lemma \ref{rem1}
(\ref{lem:BulletC}), this implies that $\overline{Q}_{0}^{s_{0}}(\mu)<\infty$.
One can show the inequality $\underline{Q}_{\flat}^{s_{0}}(\mu)>0$
in a similar manner.

\emph{ad (\ref{ConditionC}):} Assume that $q_{j}=q,j\in G_{y}$.
Let $\sigma\in\Omega_{k}$. First, we show that
\begin{eqnarray}
\mu_{\omega}=\mu_{\sigma},\;\mu_{\omega^{\flat}}=\mu_{\sigma^{\flat}},\;\;{\rm for\; all}\;\;\omega\in\Omega_{k}\;{\rm with}\;\;\sigma_{a}\prec\omega.\label{temp01}
\end{eqnarray}
In fact, one can easily see that $\mu_{\omega}=\mu_{\sigma}=p_{\sigma_{a}}q^{k-\ell(k)}$.
It remains to show that $\mu_{\omega^{\flat}}=\mu_{\sigma^{\flat}}$.
We write $\sigma_{a}=((i_{1},j_{1}),\ldots,(i_{\ell(k)},j_{\ell(k)}))$
and
\begin{eqnarray*}
\sigma & = & (\sigma_{a},j_{\ell(k+1)},\ldots,j_{k}),\;\omega=(\sigma_{a},\widetilde{j}_{\ell(k)+1},\ldots,\widetilde{j}_{k}).
\end{eqnarray*}
We have the following two cases:
\begin{itemize}
\item if $\ell(k)=\ell(k-1)$, then $\mu_{\omega^{\flat}}=\mu_{\sigma^{\flat}}=\mu_{\sigma}q^{-1}$,
since in this case we have:
\begin{eqnarray*}
\sigma^{\flat} & = & (\sigma_{a},j_{\ell(k)},\ldots,j_{k-1}),\;\omega^{\flat}=(\sigma_{a},\widetilde{j}_{\ell(k)+1},\ldots,\widetilde{j}_{k-1}).
\end{eqnarray*}

\item if $\ell(k)=\ell(k-1)+1$, by the assumption that $q_{j}=q,j\in G_{y}$, one gets
\begin{eqnarray*}
\mu_{\omega^{\flat}} & = & \mu_{\sigma^{\flat}}=q^{k-\ell(k)}p_{i_{1}j_{1}}\ldots p_{i_{\ell(k)-1}j_{\ell(k)-1}}.
\end{eqnarray*}
\end{itemize}
Next, we complete the proof for (\ref{ConditionC})
by distinguishing two cases.

\emph{Case 1:} $r>0$. For every $\omega\in\Omega_{k}$, we write
\begin{eqnarray*}
\Lambda_{C}(\omega):=\{\sigma\in\Omega_{k}:\;\omega_{a}\prec\sigma\}.
\end{eqnarray*}
Then $E\cap\bigcup_{\sigma\in\Lambda_{C}(\omega)}F_{\sigma}=E\cap E_{\omega_{a}}$.
By (\ref{temp01}), for every $\omega\in\Omega_{k}\cap\Gamma_{j,r}$,
we have $\Lambda_{C}(\omega)\subset\Gamma_{j,r}$. For each $\omega\in\Gamma_{j,r}$,
we take an arbitrary $\sigma\in\Lambda_{C}(\omega)$ and denote by
$\Gamma_{j,r}^{\flat}$ the set of such words $\sigma$. Then
\begin{eqnarray*}
\Gamma_{j,r} & = & \bigcup_{\sigma\in\Gamma_{j,r}^{\flat}}\Lambda_{C}(\sigma),\;\;\Lambda_{C}(\sigma^{(1)})\cap\Lambda_{C}(\sigma^{(2)})=\emptyset,\;\;\sigma^{(1)}\neq\sigma^{(2)}\in\Gamma_{j,r}^{\flat}.
\end{eqnarray*}
Moreover, as $\{F_{\sigma}\}_{\sigma\in\Gamma_{j,r}}$ is a cover
for $E$, we have $\sum_{\sigma\in\Gamma_{j,r}^{\flat}}\nu_{2}(E_{\sigma_{a}})=1$.
Let $\nu_{2}$ be as defined in the proof of (\ref{ConditionB}).
Then, by (\ref{maineq1}),
\[
P_{r}^{-1}Q_{r}=P_{r}^{-1}Q_{r}(P_{r}^{\theta}Q_{r}^{1-\theta})^{k}\leq P_{r}^{\ell(k)}Q_{r}^{k-\ell(k)}\leq(P_{r}^{\theta}Q_{r}^{1-\theta})^{k}=1.
\]
It follows that $P_{r}^{-1}Q_{r}P_{r}^{-\ell(k)}\leq Q_{r}^{k-\ell(k)}\leq P_{r}^{-\ell(k)}$
Note that
\[
\nu_{2}(E_{\sigma_{a}})=P_{r}^{-\ell(|\sigma|)}\big(p_{\sigma_{a}}m^{-\ell(|\sigma|)r}\big)^{\frac{s_{r}}{s_{r}+r}}.
\]
Hence, for each $\sigma\in\Gamma_{j,r}^{\flat}$, we get
\begin{eqnarray*}
\sum_{\omega\in\Lambda_{C}(\sigma)}(\mu_{\sigma}m^{-|\sigma|r})^{\frac{s_{r}}{s_{r}+r}} & = & \big(p_{\sigma_{a}}m^{-\ell(|\sigma|)r}\big)^{\frac{s_{r}}{s_{r}+r}}\bigg(\sum_{j\in G_{y}}(q_{j}m^{-r})^{\frac{s_{r}}{s_{r}+r}}\bigg)^{|\sigma|-\ell(|\sigma|)}\\
 & = & \big(p_{\sigma_{a}}m^{-\ell(|\sigma|)r}\big)^{\frac{s_{r}}{s_{r}+r}}Q_{r}^{|\sigma|-\ell(|\sigma|)}\left\{ \begin{array}{ll}
\leq\nu_{2}(E_{\sigma_{a}}),\\
\geq P_{r}^{-1}Q_{r}\nu_{2}(E_{\sigma_{a}}).
\end{array}\right.
\end{eqnarray*}
By the above analysis, we further deduce
\begin{eqnarray*}
\sum_{\sigma\in\Gamma_{j,r}}(\mu_{\sigma}m^{-|\sigma|r})^{\frac{s_{r}}{s_{r}+r}} & = & \sum_{\sigma\in\Gamma_{j,r}^{\flat}}\sum_{\omega\in\Lambda_{C}(\sigma)}(\mu_{\omega}m^{-|\omega|r})^{\frac{s_{r}}{s_{r}+r}}\\
 &  & \!\!\!\!\!\!\!\!\!\!\!\!\!\left\{ \begin{array}{ll}
\leq\,\,\,\sum_{\sigma\in\Gamma_{j,r}^{\flat}}\nu_{2}(E_{\sigma_{a}})=1,\\
\geq\,\,\, P_{r}^{-1}Q_{r}\sum_{\sigma\in\Gamma_{j,r}^{\flat}}\nu_{2}(E_{\sigma_{a}})=P_{r}^{-1}Q_{r}.
\end{array}\right.
\end{eqnarray*}
As we did in the proof of (\ref{ConditionA}), the preceding inequality
implies (\ref{main1}).

\emph{Case 2:} $r=0$. With the hypothesis of (\ref{ConditionC}),
the equation (\ref{san0}) typically does not hold. We will consider
cylinders instead of approximate squares. Set
\begin{eqnarray*}
e(\omega):=\sum_{\sigma\in\Lambda_{C}(\omega)}\mu_{\sigma}\log\mu_{\sigma},\;\omega\in\Omega^{*}.
\end{eqnarray*}
Then by (\ref{temp01}), for every $\omega\in\Lambda_{j}$, we have,
$\Lambda_{C}(\omega)\subset\Lambda_{j}$. For every $\sigma\in\Lambda_{j}$,
we take an arbitrary word $\omega\in\Lambda_{C}(\sigma)$ and denote
by $\Lambda_{j}^{\flat}$ the set of this words. Then $\Lambda_{j}=\bigcup_{\omega\in\Lambda_{j}^{\flat}}\Lambda_{C}(\omega)$.
Let $k\geq1$ and $\omega\in\Omega_{k}$. We need to show that
\begin{eqnarray}
e(\omega)=\sum_{\sigma\in\Lambda_{C}(\omega)}\sum_{\tau\in\Lambda(\sigma,h)}\mu_{\tau}\log\mu_{\tau}+\mu_{\sigma_{a}}(I_{k}-I_{k+h})\label{interimclaim}
\end{eqnarray}
for all $h\geq1$. We first prove (\ref{interimclaim}) for $h=1$. We again distinguish
two cases:

(d1): $\ell(k+1)=\ell(k)$. In this case, we have
\begin{eqnarray}
\sum_{\sigma\in\Lambda_{C}(\omega)}\sum_{\tau\in\Lambda(\sigma,1)}\mu_{\tau}\log\mu_{\tau} & = & \sum_{\sigma\in\Lambda_{C}(\omega)}\sum_{j\in G_{y}}(\mu_{\sigma}q_{j})\log(\mu_{\sigma}q_{j})\nonumber \\
 & = & \sum_{\sigma\in\Lambda_{C}(\omega)}\mu_{\sigma}\log\mu_{\sigma}+\sum_{\sigma\in\Lambda_{C}(\omega)}\mu_{\sigma}\sum_{j\in G_{y}}q_{j}\log q_{j}\nonumber \\
 & = & \sum_{\sigma\in\Lambda_{C}(\omega)}\mu_{\sigma}\log\mu_{\sigma}+\mu_{\sigma_{a}}\sum_{j\in G_{y}}q_{j}\log q_{j}\nonumber \\
 & = & e(\omega)+\mu_{\omega_{a}}(I_{k+1}-I_{k}).\label{use1}
\end{eqnarray}

(d2): $\ell(k+1)=\ell(k)+1$. Note that $E\cap\bigcup_{\sigma\in\Lambda_{C}(\omega)}F_{\sigma}=E_{\omega_{a}}\cap E$.
Thus,
\begin{eqnarray*}
e(\omega) & = & \sum_{(j_{\ell(k)+1},\ldots,j_{k})\in G_{y}^{k-l}}\prod_{h=\ell(k)+1}^{k}\mu_{\omega_{a}}q_{j_{h}}\log\bigg(\prod_{h=\ell(k)+1}^{k}\mu_{\omega_{a}}q_{j_{h}}\bigg)\\
 & = & \mu_{\omega_{a}}\log\mu_{\omega_{a}}+(k-\ell(k))\mu_{\omega_{a}}\sum_{j\in G_{y}}q_{j}\log q_{j}.
\end{eqnarray*}
Note that $\bigcup_{\sigma\in\Lambda_{C}(\omega)}\bigcup_{\tau\in\Lambda(\sigma,1)}F_{\tau}\cap E=E_{\omega_{a}}\cap E$.
We have
\begin{eqnarray*}
 &  & \sum_{\sigma\in\Lambda_{C}(\omega)}\sum_{\tau\in\Lambda(\sigma,1)}\mu_{\tau}\log\mu_{\tau}\\
 & = & \sum_{(i,j)\in G}\sum_{(j_{\ell(k)+2},\ldots,j_{k+1})\in G_{y}^{k-l}}\prod_{h=\ell(k)+2}^{k+1}\mu_{\omega_{a}}p_{ij}q_{j_{h}}\log\bigg(\prod_{h=\ell(k)+1}^{k+1}\mu_{\omega_{a}}p_{ij}q_{j_{h}}\bigg)\\
 & = & \mu_{\omega_{a}}\log\mu_{\omega_{a}}+\mu_{\omega_{a}}\sum_{(i,j)\in G}p_{ij}\log p_{ij}+(k-\ell(k))\mu_{\omega_{a}}\sum_{j\in G_{y}}q_{j}\log q_{j}.
\end{eqnarray*}
Hence, combining the above analysis, we obtain
\begin{eqnarray}
\sum_{\sigma\in\Lambda_{C}(\omega)}\sum_{\tau\in\Lambda(\sigma,1)}\mu_{\tau}\log\mu_{\tau} & = & e(\omega)+\mu_{\omega_{a}}\sum_{(i,j)\in G}p_{ij}\log p_{ij}\nonumber \\
 & = & e(\omega)+\mu_{\omega_{a}}(I_{k+1}-I_{k}).\label{use2}
\end{eqnarray}
Combining (\ref{use1}) and (\ref{use2}), for $k\geq1$ and all $\omega\in\Omega_{k}$,
we have
\begin{eqnarray}
e(\omega) & = & \sum_{\sigma\in\Lambda_{C}(\omega)}\sum_{\tau\in\Lambda(\sigma,1)}\mu_{\tau}\log\mu_{\tau}+\mu_{\sigma_{a}}(I_{k}-I_{k+1}).\label{san4}
\end{eqnarray}
Assume that (\ref{interimclaim}) holds for $h=p\in\mathbb{N}$, by
(\ref{san4}) and mathematical induction, one can show that $(\ref{interimclaim})$
holds for all $h\in\mathbb{N}$, which is equivalent to
\[
\sum_{\sigma\in\Lambda_{C}(\omega)}\mu_{\sigma}(\log\mu_{\sigma}-I_{k})=\sum_{\sigma\in\Lambda_{C}(\omega)}\sum_{\tau\in\Lambda(\sigma,h)}\mu_{\tau}(\log\mu_{\tau}-I_{k+h}).
\]
Applying the preceding equation to every $\sigma\in\Lambda_{j}$ with
$h=k_{2j}-|\sigma|$, one gets
\begin{eqnarray*}
\sum_{\sigma\in\Lambda_{j}}\mu_{\sigma}(\log\mu_{\sigma}-I_{|\sigma|}) & = & \sum_{\omega\in\Lambda_{j}^{\flat}}\sum_{\sigma\in\Lambda_{C}(\omega)}\mu_{\sigma}(\log\mu_{\sigma}-I_{|\sigma|})\\
 & = & \sum_{\omega\in\Lambda_{j}^{\flat}}\bigg(\sum_{\sigma\in\Lambda_{C}(\omega)}\sum_{\tau\in\Lambda(\sigma,k_{2j}-|\sigma|)}\mu_{\tau}\log\mu_{\tau}-\mu_{\omega_{a}}I_{k_{2j}}\bigg)\\
 & = & \sum_{\tau\in\Omega_{k_{2j}}}\mu_{\tau}\log\mu_{\tau}-I_{k_{2j}}=0
\end{eqnarray*}
It follows that $\sum_{\sigma\in\Lambda_{j}}\mu_{\sigma}\log\mu_{\sigma}=\sum_{\sigma\in\Lambda_{j}}\mu_{\sigma}I_{|\sigma|}$.
Hence, we obtain
\begin{eqnarray*}
t_{j} & = & \frac{\sum_{\sigma\in\Lambda_{j}}\mu_{\sigma}\log\mu_{\sigma}}{\sum_{\sigma\in\Lambda_{j}}\mu_{\sigma}\log m^{-|\sigma|}}=\frac{\sum_{\sigma\in\Lambda_{j}}\mu_{\sigma}I_{|\sigma|}}{\sum_{\sigma\in\Lambda_{j}}\mu_{\sigma}\log m^{-|\sigma|}}
\end{eqnarray*}
Thus, (\ref{san5}) holds and (\ref{main1}) follows by the last part
of the proof for (\ref{ConditionB}).
This finishes the proof of the main theorem.

\subsection{Concluding remarks and examples}

As an immediate consequence of Theorem \ref{mthm1} (\ref{ConditionB})
and (\ref{ConditionC}), we have
\begin{corollary}
\label{cor:Permutation}Assume that, if the vectors $(p_{ij}/q_{j})_{j\in G_{x,j}},j\in G_{y}$
are permutations of one another. Then (\ref{main1}) holds for all
$r\geq0$.
\end{corollary}
Next, we construct an example to illustrate Theorem \ref{mthm1}.
\begin{example}{\rm
Let $n=9,m=3$. We consider the functions
\begin{eqnarray*}
g_{1}(x):=\sqrt{x}+\sqrt{\frac{3}{8}-x},\; x\in[0,\frac{3}{8}],\\
g_{2}(x):=\sqrt{x}+\sqrt{\frac{7}{16}-x},\; x\in[0,\frac{7}{16}].
\end{eqnarray*}
Then $g_{1},g_{2}$ are both continuous. Note that
\begin{eqnarray*}
 &  & g_{1}(0)=\frac{\sqrt{6}}{4}<\frac{\sqrt{2}}{2},\; g_{1}\left(\frac{3}{16}\right)=\frac{\sqrt{3}}{2}>\frac{\sqrt{2}}{2};\\
 &  & g_{2}(0)=\frac{\sqrt{7}}{4}<\frac{\sqrt{2}}{2},\; g_{1}\left(\frac{7}{32}\right)=\frac{\sqrt{14}}{4}>\frac{\sqrt{2}}{2}.
\end{eqnarray*}
The continuity of $g_{1}$ allows us to choose a real number $x_{1}\in(0,\frac{3}{8})$
satisfying
\[
\sqrt{\frac{1}{8}}+\sqrt{x_{1}}+\sqrt{\frac{3}{8}-x_{1}}=\frac{3\sqrt{2}}{4}.
\]
In a similar manner, one can find a number $x_{2}\in(0,\frac{7}{16})$
such that
\begin{equation}
\sqrt{x_{2}}+\sqrt{\frac{7}{16}-x_{2}}=\frac{\sqrt{2}}{2}.\label{x1}
\end{equation}
Let $G:=\{(1,0),(3,0),(5,0),(1,2),(3,2),(5,2),(7,2)\}$. We set
\begin{eqnarray}
 &  & p_{10}=\frac{1}{8},\; p_{30}=x_{1},\; p_{50}=\frac{3}{8}-x_{1};\nonumber \\
 &  & p_{12}=x_{2},p_{32}=\frac{7}{16}-x_{2},\; p_{52}=\frac{1}{32},\; p_{72}=\frac{1}{32}.\label{x2}
\end{eqnarray}
Then one can easily see that $q_{0}=q_{2}=\frac{1}{2}$. Let $\mu$
be the self-affine measure associated with $G$ and $\left(p_{ij}\right)_{(i,j)\in G}$.
In view of (\ref{x1}) and (\ref{x2}), for $r=1$,
\begin{eqnarray}
C_{2,r}=\sum_{i\in G_{x,2}}\bigg(\frac{p_{i2}}{q_{2}}\bigg)^{\frac{1}{2}}=C_{0,r}=\sum_{i\in G_{x,0}}\bigg(\frac{p_{i0}}{q_{0}}\bigg)^{\frac{1}{2}}=\sqrt{2}\sum_{i\in G_{x,0}}(p_{i0})^{\frac{1}{2}}=\frac{3}{2}.\label{hyporeal}
\end{eqnarray}
So, for $r=1$ and $G_{y}=\{0,2\},G_{x,0}=\{1,3,5\}$ and $G_{x,2}=\{1,3,5,7\}$,
we have
\begin{eqnarray*}
\sum_{j\in G_{y}}(q_{j}m^{-r})^{\frac{1}{r+1}}\bigg(\sum_{i\in G_{x,j}}\bigg(\frac{p_{ij}}{q_{j}}\bigg)^{\frac{1}{r+1}}\bigg)^{\log_{9}^{3}}=2\bigg(\frac{\sqrt{6}}{6}\cdot\frac{\sqrt{6}}{2}\bigg)=1.
\end{eqnarray*}
Hence, for $r=1$, we have that $s_{r}=t_{r}=1$. By (\ref{hyporeal}),
the hypothesis in Theorem \ref{mthm1} (\ref{ConditionA}) is satisfied.
Thus, we conclude that (\ref{main1}) holds for $r=1$. Finally,
since $q_{0}=q_{2}=2^{-1}$, by Theorem \ref{mthm1} (\ref{ConditionC}),
(\ref{main1}) holds for all $r\geq0$.
}\end{example}

\noindent{\bf Acknowledgement} SZ is supported by the \emph{Chinese Scholarship Council} - File
No. 201308320049.

\end{document}